\documentclass[11pt]{article}
\input epsf.tex
\usepackage{amsthm, amsmath, amssymb,
amscd, amsfonts, latexsym, verbatim}
\usepackage{enumerate}

\DeclareFontFamily{OT1}{wncyr}{\hyphenchar\font45}
\DeclareFontShape{OT1}{wncyr}{m}{n}{%
   <5> <6> <7> <8> <9> gen * wncyr
   <10> <10.95> <12> <14.4> <17.28> <20.74>  <24.88>wncyr10}{}
\DeclareFontShape{OT1}{wncyr}{m}{it}{%
   <5> <6> <7> <8> <9> gen * wncyi
   <10> <10.95> <12> <14.4> <17.28> <20.74> <24.88> wncyi10}{}
\DeclareFontShape{OT1}{wncyr}{m}{sc}{%
   <5> <6> <7> <8> <9> <10> <10.95> <12> <14.4>
   <17.28> <20.74> <24.88>wncysc10}{}
\DeclareFontShape{OT1}{wncyr}{b}{n}{%
   <5> <6> <7> <8> <9> gen * wncyb
   <10> <10.95> <12> <14.4> <17.28> <20.74> <24.88>wncyb10}{}
\input cyracc.def

\DeclareMathSizes{9}{9}{7}{5} 

\newtheorem{thm}{Theorem}
\newtheorem{cor}{Corollary}
\newtheorem{prop}{Proposition}

\theoremstyle{remark}

\theoremstyle{definition}

\newtheorem{example}{\bf Example}

\title{\bfseries The cone of Hilbert Nullforms}
\author{\Large Vladimir~L.~Popov\\
\footnotesize
Steklov Mathematical Institute\\[-5pt]
\footnotesize
Russian Academy of Sciences \\[-5pt]
\footnotesize Gubkina 8, Moscow, 119991, Russia\\
{ \footnotesize \tt popovvl@mi.ras.ru} }
\date{}

\begin{document}

\maketitle

\begin{abstract}
We describe a geometric--combinatorial algorithm that allows one, using solely the system of weights and roots, to determine the Hesselink strata of the null-cone of a linear representation of a reductive algebraic group and calculate their dimensions. In particular, it provides a constructive approach to calculating the dimension of the null-cone and determining all its irreducible components of maximal dimension. In the case of the adjoint representation (and, more generally, a $\theta$-representation), the algorithm turns into the classification algorithm for the conjugacy classes of nilpotent elements in a semisimple Lie algebra (respectively, homogeneous nilpotent elements in a cyclically graded semisimple Lie algebra).
\end{abstract}

\section*{1. Introduction\label{vved}}

Below all algebraic varieties are taken over an algebraically closed field $k$ of characteristic zero.

Let $G$
 be a reductive algebraic group and let $V$
be a finite-dimensional algebraic
$G$-module.
A vector $v\in V$  is called a {\it Hilbert nullform} (cf.\;\cite{H2};  other terms are
{\it unstable} vector, \cite{M}, {\it
nilpotent} vector, \cite{PV}), if any $G$-invariant polynomial function $f$ on $V$ that vanishes at
$0$ also vanishes at $v$.  The closed subscheme ${\cal N}^{}_{G, V}$ of $V$
defined by vanishing of all such
$f$ is called the {\it null-cone of} $V$. Hilbert was the first who found that the null-cones play a fundamental role in the theory of invariants and its applications, see\,\cite{Hi}. In the general case, the null-cone has a complicated structure; it may be irreducible, reducible, or nonequidimensional, and may have complex singularities.

According to the Hilbert--Mumford criterion,
\cite{Hi},\,\cite{M}, a vector $v\in V$
is a nullform if and only if there is a one-parameter subgroup   $\lambda: k^\times\rightarrow G$ that steers
 $v$ to $0$ (i.e., such that
$\operatorname{lim}_{t\rightarrow
0}\lambda(t)\cdot v=0$). In 1978, when solving different problems, several authors simultaneously revealed the key role of those $\lambda$
 that do it ``most optimally'', see \cite{B}, \cite{Ke}, \cite{H1},
\cite{H2}, \cite{R} (however, for
$V=\operatorname{Lie}(G)$, where the null-cone coincides with the cone of nilpotent elements, this fact was revealed in a different form
already
in
\cite{D}: indeed, it was found out later, see
\cite{H2}, \cite{H1}, \cite{Sl},  that the ``characteristics'' of nilpotent elements introduced in
\cite{D} can be equivalently defined in terms of the optimality property). The formalization of this property is performed by introducing an
$\operatorname{Ad}(G)$-invariant norm on the set of all one-parameter subgroups: optimal subgroups are those that have the minimal norm. Technically, it is more convenient to consider a wider class of virtual one-parameter subgroups, \cite{H1},
\cite{H2},  which is actually the set
$\operatorname{Lie}(G)(\mathbb Q)$
 of rational semisimple elements in the Lie algebra  $\operatorname{Lie}(G)$.  The norm is defined by the choice of an
$\operatorname{Ad}(G)$-invariant inner product $\langle \ {,}\
\rangle$ on $\operatorname{Lie}(G)$
 that is rational-valued and
 positive definite on the $\mathbb
Q$-vector space
$\operatorname{Lie}(T)(\mathbb Q)$  for a certain (and thereby for any) maximal torus $T$ in $G$. If
$\Lambda(v)$
is a family of all optimal virtual one-parameter subgroups corresponding to a nullform ~$v$, then, following \cite{H2},  we can consider the following equivalence relations on the set of all nullforms:
$$
x\sim y \ \Longleftrightarrow \
\Lambda(g\cdot x) =\Lambda(y)\ \text{ for a certain } \ g\in G.
$$
The equivalence classes form a finite partition of the null-cone ${\cal N}^{}_{G, V}$
and are called {\it strata}. Note that, in general, the strata depend on the choice of the norm. It is remarkable that, although the null-cone
${\cal N}^{}_{G, V}$ itself has a complicated structure, the configuration of the strata turns out to be simple: Hesselink \cite{H2} proved that each stratum $S$ is an irreducible subvariety of the null-cone ${\cal
N}^{}_{G, V}$,  open in its closure
$\overline{S}$, and there exists an isomorphism
$\varphi$ of the stratum $S$ onto an invariant open subset of a homogeneous vector bundle $E$ over $G/P$, where $P$ is a parabolic subgroup in $G$ (all these objects depend on $S$); in particular, $S$
is a smooth rational variety. Moreover, there exists a morphism $\pi:
E\rightarrow \overline{S}$,  which is a resolution of singularities of the variety
$\overline{S}$ and is such that
$\pi|^{}_{\pi^{-1}(S)}=\varphi^{-1}$. In particular, this gives a resolution of singularities for each irreducible component of the null-cone
${\cal N}^{}_{G, V}$ and shows that, using Kempf's terminology, \cite{Ke}, the latter is obtained by collapsing of a homogeneous vector bundle over a generalized flag variety of the group $G$.  Note that, in general, the closure of a stratum is not a union of strata.

Developing the idea outlined in
\cite[(2.2.22)]{P3}
we consider in this paper the problem of
constructive
description of strata.
This problem has been already discussed in т\;\cite[Sec\-tion 5]{H2}, but has not been solved there: concerning the directly related
problem
of finding out if a given saturated set is the closure of a blade (see \cite{H2} for the meaning of these terms; we will not use them any longer), Hesselink 
 writes that it ``is more difficult'' and
  restricts the analysis to a certain necessary condition. In the present paper 
  we show that the family of all strata of the null-cone
${\cal N}^{}_{G, V}$  is completely determined only by the geometric configuration of two finite subsets of the Euclidean space
$\operatorname{Lie}(T)(\mathbb Q)$, namely, the system of weights of the $G$-module $V$ (endowed with their multiplicities) and the system of roots of the group $G$,
and we obtain a simple geometric-combinatorial algorithm for describing this family; this algorithm employs only elementary geometric operations, namely,  the orthogonal projection of a finite system of points onto a linear space and taking its convex hull. Technically, the procedure is reduced to constructing certain oriented rooted trees by means of these operations and assigning plus or minus signs to the vertices of these trees according to a certain rule. The strata are defined precisely by the trees to whose roots
is assigned a plus sign.

More precisely, this algorithm allows one to describe constructively
the following objects:

\begin{list}{}{\setlength{\topsep}{0.8mm}
\setlength{\parsep}{-1mm}
\setlength{\itemindent}{4mm}} \item[(a)]
 a certain finite subset $\mathfrak
S$ in $\operatorname{Lie}(T)(\mathbb Q)$;
\item[(b)] a certain linear subspace $V[l^+]$ in $V$ for every element $l\in \mathfrak S$; and \item[(c)]
a certain open subset $V[l^+]^0$ in
$V[l^+]$,
\end{list}

\noindent such that the sets
 $G\cdot V[l^+]^0$,  where $l$ runs over $\mathfrak S$,
exhaust, without repetition, the family of all strata of the null-cone ${\cal N}^{}_{G, V}$.

In this case, the dimension of the stratum $G\cdot V[l^+]^0$
 turns out to be equal to the sum of the number of roots lying the half-space
 $\{v \in \operatorname{Lie}(T)(\mathbb
 Q)\mid \langle l,v \rangle <0 \}$  and the number of weights (taken with their multiplicities) lying the half-space
 $\{v\in \operatorname{Lie}(T)(\mathbb
 Q)\mid \langle l,v\rangle \geqslant 1\}$.
  This yields a simple constructive approach to calculating the dimension of the null-cone ${\cal N}^{}_{G, V}$:
 it is equal to the maximum of the above sums taken over all $l$ from $\mathfrak S$.
 Calculating $\dim {\cal N}^{}_{G, V}$ is
 a nontrivial problem;
 for example, the problem of classifying equidimensional representations and that of classifying representations with a free module of covariants, \cite{P1}, \cite{Ad}, \cite{S1},
 \cite{L1},  are reduced to the above problem; their solutions
 are
 actually based on
 developing various ad hoc methods for estimating $\dim {\cal N}^{}_{G, V}$. Until now, general information on $\dim {\cal N}^{}_{G, V}$
 has been restricted to various a priori bounds of  this number
 based on the Hilbert--Mumford criterion, see\,\cite{P1},
\cite{S1}, \cite{S2}, \cite{Ad}, \cite{K2}.

Moreover, this also yields a simple and constructive method for describing all irreducible components of maximal dimension of the null-cone
${\cal N}^{}_{G, V}$: they are exhausted, without repetition, by varieties $G\cdot V[l^+]$, where $l$ is an element for which the above maximum is attained.

The algorithm can be applied for composing a computer program for the constructive description of all strata, calculating their dimensions, calculating the dimension of the null-cone, and describing all its irreducible components of maximal dimension.
The operating rate of such a program is determined by its part related to sorting out all the subsets of the set of weights. 
 For $G=T$, the system of weights may be any finite set of characters, so such looking over is inevitable. But if $G$ is noncommutative, the Weyl group provides nontrivial symmetries of the system of weights, which can be used for improving the enumerative part of the program.\footnote{{\it Added in September} 2010:  Jointly with N.\,A'Campo a computer implementation of this algorithm
 has been composed in 2005.~Pro tempore the code can be obtained on request sent to the address {\tt norbertacampo@gmail.com}. But we hope to compose also the relevant comments for publishing  and posting both on the web.
}

In the case of the adjoint representation, and, more generally, of a so-called $\theta$-represen\-ta\-tion (see
\cite{V1}, \cite{K2}, as well as \cite[8.5]{PV}
 and \cite{L1}) our algorithm turns into the algorithm of classifying conjugacy classes of nilpotent elements in a semisimple Lie algebra and, respectively, of homogeneous nilpotent elements in a cyclically graded semisimple Lie algebra (see below Proposition 5).
Although these classifications, obtained in
\cite{D} and
\cite{V1}, \cite{V2} by a different method, are known, it would be interesting to derive (and check) them again in nonserial cases by a computer program based on our algorithm. This, in particular, applies to the classification of nullforms in the spaces of 3-vectors of a 9-dimensional space,
\cite{VE}, 4-vectors of an 8-dimensional space, \cite{An}, spinors of 14- and 16-dimensional spaces, \cite{P2},
\cite{AE}, and to certain other classifications. In fact, each of the cited results can be considered as the solution,  in the case under consideration, of the classification problem for the strata of the null-cone. Let us mention another important case when the classification problem has also been solved,
\cite{LB}: it is the case of the diagonal action  of the group
$\operatorname{GL}_n$  by conjugation on the space of  $m$-tuples of $n\times n$-matrices.

The final section contains examples of application of the algorithm. We discuss the cases of rank
$\leqslant 2$, ternary forms, and nilpotent elements in semisimple Lie algebras (the case of a special simple Lie algebra of type
$\operatorname{G}_2$ is considered in detail) and in spaces of
$\theta$-representations. In this section we also analyze an example of a certain action of
$\operatorname{GL}_2$ which shows that, in general, strata depend on the choice of the norm and may not coincide with orbits if there are only finitely many orbits in the null-cone.

\section*{2. Strata}

Since the nullforms for the actions of the group $G$ and its identity component in $V$ are the same, we will assume that
$G$  is a {\it connected} group. The space $V$
is assumed to be {\it nonzero}.

\smallskip

{\bf 2.1.} \hskip 1.5mm First, we shall introduce certain notation and conventions.

Let $S$  be an algebraic torus. The group
${\rm Hom}(S,k^{\times})$ of its rational characters is a free abelian group of finite rank. We consider it in additive notation and denote by
$s^\mu$ the value of the character $\mu\in {\rm
Hom}(S,k^{\times})$ at point~$s\in S$.

Let $\mathfrak s:={\rm Lie}(S)$.
 The map ${\rm
Hom}(S,k^{\times})\rightarrow \mathfrak s^*,
\ \mu\mapsto d_e\mu $, is an embedding. We identify ${\rm Hom}(S,k^{\times})$ with the image of this embedding. Then, ${\rm
Hom}(S,k^{\times})$
 is a lattice of full rank in $\mathfrak s^*$,
so that its linear span over $\mathbb Q$
is a  $\mathbb Q$-form of the vector space $\mathfrak s^*$. Let
$\mathfrak s(\mathbb Q)$ be the
$\mathbb Q$-form of the Lie algebra $\mathfrak s$ defined by this span:
\begin{equation}
\mathfrak s(\mathbb Q):=\{s\in \mathfrak s\mid
l(s)\in \mathbb Q\ \forall \   l\in {\rm
Hom}(S,k^{\times})\}. \label{rational}
\end{equation}

For any algebraic $S$-module $L$ and character $\mu\in {\rm Hom}(S,k^{\times})$
we set
$$
L_\mu:=\{l\in L\mid s\cdot l=s^\mu l \
\forall\ s\in S\}
$$
and denote by $\Delta^{}_{L, S}$ the system of weights of $S$ in $L$,
$$
\Delta^{}_{L, S}:=\{\mu\in {\rm
Hom}(S,k^{\times})\mid L_\mu\neq 0\}.
$$
Then, $ L=\bigoplus_{\mu\in \Delta^{}_{L,
S}}L_\mu.$  The multiplicity of the weight $\mu$ in the
$S$-module $L$
is the number $\dim L_\mu$.

\smallskip

{\bf 2.2.} \hskip 1.5mm Let $T$ be a maximal torus in $G$ and $\mathfrak t:={\rm
Lie}(T)$. Fix an ${\rm
Ad}(G)$-invariant inner product
$\langle\ {,}\ \rangle$ on $\mathfrak
g:={\rm Lie}(G)$ whose restriction on
$\mathfrak t(\mathbb Q)$
is rational-valued and
positive definite. The Weyl group $W^{}_{G, T}:={\rm
Nor}_{G}(T)/T$, considered as a group of transformations of the space $\mathfrak t$, carries
$\mathfrak t(\mathbb Q)$ into itself, and the restriction $\langle\ {,}\
\rangle|^{}_{\mathfrak t(\mathbb Q)}$
is $W^{}_{G, T}$-invariant. This restriction uniquely defines the inner product $\langle\ {,}\ \rangle$.  Any rational-valued positive definite $W^{}_{G,
T}$-invariant inner product on
$\mathfrak t(\mathbb Q)$ is obtained as such a restriction. In what follows, we assume that the strata of the null-cone  ${\cal
N}^{}_{G, V}$ are defined by the norm on
$\operatorname{Lie}(G)(\mathbb Q)$, which is defined by the inner product $\langle\ {,}\
\rangle$.

We identify $\mathfrak t$ with $\mathfrak
t^*$ by means of the map $\mathfrak t\rightarrow
\mathfrak t^*, \ t\mapsto \langle
t,\cdot\rangle$. Then ${\rm
Hom}(T,k^{\times})$ becomes a lattice of full rank in $\mathfrak t(\mathbb Q)$.  Henceforth, we consider $\mathfrak t(\mathbb Q)$
as a Euclidean linear space over $\mathbb Q$ with respect to the inner product $\langle\ {,}\
\rangle|^{}_{\mathfrak t(\mathbb Q)}$.

\smallskip

{\bf 2.3.} We have the following decomposition:
\begin{equation}
 V=\bigoplus_{\mu\in \Delta^{}_{V, T}} V_\mu.
\label{v}
\end{equation}

Since maximal tori are conjugate and their union is dense in $G$, the action of $G$ on
$V$ is trivial if and only if
$\Delta^{}_{V, T}=\{0\}$.

If $V$ is the adjoint $G$-module
$\mathfrak g$, then $\Delta^{}_{\mathfrak g,
T}=\Phi_{\mathfrak g, T}\sqcup \{0\}$, where
$\Phi_{\mathfrak g, T}$
is the system of roots of the Lie algebra $\mathfrak g$
with respect to the torus $T$. We have $\dim
\mathfrak g_\alpha=1\ \forall\ \alpha\in
\Phi_{\mathfrak g, T}$, $\mathfrak
g_0=\mathfrak t$ and
\begin{equation}  
\mathfrak g=\mathfrak t\oplus
\bigoplus_{\alpha\in \Phi_{\mathfrak g,
T}}\mathfrak g_{\alpha}. \label{root}
\end{equation}

{\bf 2.4.} \hskip 1.5mm Let $l\in\mathfrak
t(\mathbb Q)$
 be a nonzero element. For any number $\gamma
\in \mathbb Q$ we set
$$\{l=\gamma\}:=\{t\in \mathfrak t(\mathbb
Q)\mid \langle l, t\rangle=\gamma\}.$$
$\{l\geqslant
\gamma\}$, etc. are defined similarly.

The element $l$ determines in $\mathfrak g$
the following algebraic subalgebras:
 \begin
 {equation}
  \mathfrak t[l]:=\{
  t\in \mathfrak
 t \mid (l, t)=0\},
 \quad
 \mathfrak g[l]:=\mathfrak t[l]\oplus
 \bigoplus_{
 \alpha\in\{l=0\}}
 \mathfrak g_{\alpha},
 \quad
 \mathfrak p[l]:= \mathfrak t\oplus
 \bigoplus_{
 \alpha\in \{l\geqslant0\}}
 \mathfrak g_{\alpha}.
 \label{subalg}
\end{equation}

Let $T[l]$, $G[l]$ and $P[l]$
be connected algebraic subgroups of the group
$G$ with the Lie algebras $\mathfrak t[l]$,
$\mathfrak g[l]$ and $\mathfrak p[l]$, respectively.
$G[l]$  is a reductive group, $T[l]$ is its maximal torus, and
$P[l]$ is a parabolic subgroup in $G$. It follows from
\eqref{rational} and \eqref{subalg} that
\begin{equation}
\mathfrak t[l](\mathbb Q)=\{l=0\}, \quad
\Phi_{\mathfrak g[l], T[l]}=\Phi_{\mathfrak
g, T}\cap {\mathfrak t}[l]({\mathbb Q}),
\quad {\mathfrak g}[l]_{\alpha}={\mathfrak
g}_{\alpha} \ \forall \ \alpha\in
\Phi_{{\mathfrak g}[l], T[l]}.
\label{subtorus}
\end{equation}

The element $l$ also determines the following linear subspaces in $V$:
\begin{equation}
V[l]:= \hskip -3mm\bigoplus_{ \mu\in\{l=1\}}
V_{\mu},
\quad
V[l^{+}]:= \bigoplus_{
\mu\in\{l\geqslant 1\}} V_{\mu },
\label{vect}
\end{equation}

Relations \eqref{subalg} and \eqref{vect} imply the
$G[l]$-invariance of $V[l]$ and the
$P[l]$-invariance of $V[l^+]$.

\smallskip

{\bf 2.5.}\hskip 1.5mm Let
\begin{equation*}
V[l^+]\rightarrow V[l], \ v\mapsto v',
\label{proj}
\end{equation*}
be the projection parallel to
$\bigoplus_{\mu\in\{l>1\}} V_{\mu}$.
Consider in $V[l]$ the null-cone ${\cal
N}_{G[l], V[l]}$    and in $V[l^+]$ the open
$P[l]$-invariant subset
\begin{equation}
V[l^+]^0:= \bigl\{v\in V[l^+]\mid v'\notin
{\cal N}^{}_{G[l],V[l]}\bigr\}
\label{fiber}
\end{equation}

In general, the set  $V[l^+]^0$  may be empty. By \eqref{fiber}, its noneptiness is equivalent to the inequality
\begin{equation}{\cal N}^{}_{G[l], V[l]}\neq
V[l]. \label{strat}\end{equation}

\smallskip

\noindent {\bf Definition.}   A nonzero element $l\in \mathfrak t(\mathbb Q)$
is called a {\it stratifying}
element if condition \eqref{strat} holds.

\medskip

{\bf 2.6.}\hskip 1.5mm Suppose that $l$
is a stratifying element from
$\mathfrak t(\mathbb Q)$. Then, the variety $V[l^+]^0$  is nonempty. Let
$G\times^{P[l]}\!V[l^+]^0$  be the space of the homogeneous bundle over $G/P[l]$ with the fiber
$V[l^+]^0$. Consider the morphism
\begin{equation}
G\times^{P[l]}\!V[l^+]^0\longrightarrow V, \ \
[g, v]\mapsto g\cdot v, \label{iso}
\end{equation}
where $[g, v]$ is the image of the point $g\times v$ under the quotient morphism  $G\times V[l^+]^0\rightarrow
G\times^{P[l]}\!V[l^+]^0$.  Let ${\cal H}[l]$
be the image of morphism
\eqref{iso},
\begin{equation} {\cal H}[l]:=G\cdot
V[l^+]^0. \label{h}
\end{equation}

\begin{thm}\label{stratum}\

 {\rm (1)} If
$l\in \mathfrak t(\mathbb Q)$ is a stratifying element, then
\begin{list}{}{\setlength{\topsep}{1.6mm}\setlength{\parsep}{-1mm}
\setlength{\itemindent}{2.5mm}} \item[\rm
(a)] ${\cal H}[l]$ is a stratum of the null-cone
 ${\cal N}_{G,
V}$ that is different from $\{0\}$; \item[\rm (b)] morphism {\rm
\eqref{iso}} isomorphically maps
$G\times^{P[l]}\!V[l^+]^0$ onto ${\cal H}[l]$;
\item[\rm (c)] $G\cdot V[l^+]$ is the closure of ${\cal H}[l]$.
\end{list}

{\rm (2)} Any stratum of the null-cone ${\cal N}_{G, V}$ that is different from $\{0\}$ has the form
${\cal H}[l]$ for a certain stratifying element  $l$.

{\rm (3)} ${\cal H}[l_1]={\cal H}[l_2]$
if and only if $l_1$ and $l_2$
lie on the same $W^{}_{G, T}$-orbit.
\end{thm}
\begin{proof} See \cite[Theorem
5.6]{PV} (see also \cite[Prop.\,1]{Sl},
\cite[Proof of Theorem 9.2]{N}, \cite[Remark
12.21]{Ki}).
\end{proof}

Since $\dim G\times^{P[l]}\!V[l^+]^0= \dim
G/P[l]+\dim V[l^+]^0,$ formulas \eqref{v},
\eqref{root}, \eqref{subalg}, \eqref{vect},
\eqref{fiber}, claim (1)(b) of Theorem
\ref{stratum} and the one-dimensionality of the root subspaces $\mathfrak g_\alpha$ imply the following corollary.
\begin{cor}[\rm formula for the dimension of a stratum] \label{dimen}
\begin{equation} 
\dim {\cal H}[l]= \#
\{l<0\}\cap\Phi_{\mathfrak g, T}\ +\hskip
-1mm \sum_{ \mu\in \{l\geqslant 1 \} }\hskip
-1mm \dim V_{\mu}. \label{dim}
\end{equation}
\end{cor}

In turn, since the strata are open in their closures and the number of strata is finite, we obtain the following corollary from  \eqref{v} and \eqref{dim}.
\begin{cor}[\rm numerical criterion for the openness of a stratum in $V$] \label{equal} The following properties of a stratifying element $l\in \mathfrak
t(\mathbb Q)$ are equivalent:

\begin{list}{}{\setlength{\topsep}{1.6mm}\setlength{\parsep}{-1mm}
\setlength{\itemindent}{2.5mm}} \item[\rm
(a)] the stratum ${\cal H}[l]$ is open in $V$,
\item[\rm (b)] $\dim {\cal H}[l]=\dim V$,
\item[\rm (c)]  $\#\{l<0\}
\cap\Phi_{\mathfrak g, T} = \sum_{\mu\in
\{l<1\} }\dim V_{\mu}.$
\end{list}

A stratifying element $l$ with these properties exists if and only if
${\cal N}_{G, V}=V$.
\end{cor}

Finally, the inclusion ${\cal H}(l)\subseteq
V$ yields the following corollary.

\begin{cor}
If an element $l$ of $\mathfrak t(\mathbb Q)$
is a stratifying element, then
\begin{equation}  
\#\{l<0\} \cap\Phi_{\mathfrak g, T}\
\leqslant\hskip -1mm\sum_{\mu\in \{l<
1\}}\hskip -1mm\dim V_{\mu}. \label{ineq}
\end{equation}
\end{cor}

{\bf 2.7.}\hskip 1.5mm  In the next section we will show that the family of all stratifying elements in $\mathfrak
t(\mathbb Q)$ is completely determined only by the geometric configuration of the system of weights
 $\Delta^{}_{V, T}$ (endowed with their multiplicities) and the system of roots $\Phi_{\mathfrak g, T}$
 in the Euclidean space $\mathfrak t(\mathbb Q)$, and obtain a simple geometric--combinatorial algorithm for describing this family.

 Namely, this algorithm allows one to describe constructively a finite set $\mathfrak
S$ of nonzero elements
 $l\in \mathfrak t(\mathbb Q)$ such that
 $V[l^+]\neq\varnothing $ for every $l\in
 \mathfrak S$ and that the sets $G\cdot V[l^+]^0$,
 where $l$ runs over $\mathfrak S$,
 exhaust, without repetition, the family of all strata in
 ${\cal N}^{}_{G, V}$ that are different from $\{0\}$.
By Theorem~\ref{stratum} this yields the following description for the null-cone:
\begin{equation}
{\cal N}^{}_{G, V}=\bigcup_{l\in \mathfrak
S}G\cdot V[l^+]. \label{nulvar}
\end{equation}
This description is constructive in the sense that our algorithm constructively describes all linear subspaces $V[l^+]$ from
\eqref{nulvar}. Since the application of this algorithm to the $G[l]$-module $V[l]$ for $l\in
\mathfrak S$ yields a constructive description, analogous to \eqref{nulvar},
 of the null-cone ${\cal
N}^{}_{G[l], V[l]}$, we obtain from
\eqref{fiber} a constructive description for every stratum $G\cdot V[l^+]^0$.

In view of Corollary \ref{dimen}  to Theorem
\ref{stratum}, the constructive description of the set $\mathfrak S$ also yields the following constructive method for calculating the dimension of the null-cone ${\cal N}^{}_{G, V}$
by means of calculating the number of roots and weights (with multiplicities) in appropriate half-spaces of the space
$\mathfrak t(\mathbb Q)$:
 \begin{prop}
 \begin{equation}
 \dim {\cal N}^{}_{G, V}=\max_{l\in\mathfrak S}
\bigl\{\#\{l<0\}\cap\Phi_{\mathfrak g, T} \
+\hskip -1mm \sum_{ \mu\in \{l\geqslant 1 \}
}\hskip -1mm \dim V_{\mu}\bigr\}.
 \label{dimzero}
 \end{equation}
 \end{prop}
 In addition, Corollary \ref{dimen} to
Theorem \ref{stratum}  yields a constructive method for describing all irreducible components of maximal dimension of the null-cone
${\cal N}^{}_{G, V}$: these components are exhausted, without repetition, by the varieties $G\cdot V[l^+]$, where $l\in\mathfrak S$ is an element for which the sum
  $\#\{l<0\}\cap\Phi_{\mathfrak g, T} \
+\hskip -1mm \sum_{ \mu\in \{l\geqslant 1 \}
}\hskip -1mm \dim V_{\mu}$
attains its maximum.

 \medskip

 {\bf 2.8.}\hskip 1.5mm In some important cases, it is a priori known that the strata are orbits. For example, this is so for adjoint representations and, more generally, for $\theta$-representations (see below Proposition 4). In these cases, the number of orbits in the null-cone
 ${\cal N}^{}_{G, V}$ is finite, and our algorithm turns into the classification algorithm for these orbits. The following proposition can be used (see below Example 5) for describing the representatives of such orbits.

 \begin{prop}\label{represent} Consider a
 $T$-weight basis
 $\{v_i\}_{i\in I}$ in $V$ and a set
 $\{X_\alpha\in \mathfrak g_{\alpha} \mid
 X_\alpha\neq 0, \alpha\in \Phi_{\mathfrak g, T}\}$
such that, for any $i$ and
 $\alpha$, the
 coordinates of the vector $X_{\alpha}\cdot v_i$
in the basis $\{v_i\}_{i\in I}$
lie in $\mathbb Q$.
Let $l$ be a stratifying element from~$\mathfrak t(\mathbb Q)$ and $J:=\{i\in
I\mid v_i\in V[l]\}$. Then, for any family of constants
 $\{c_j\in k\mid j\in
J\}$ that are algebraically independent over
 $\mathbb Q$, the vector $v:=\sum_{j\in J} c_jv_j$ lies in the stratum~${\cal H}^{}_{G, V}$.
\end{prop}
\begin{proof} By construction, $v\in V[l]$.
By \eqref{fiber} and Theorem \ref{stratum}
the proof boils down to checking that $v\notin
{\cal N}^{}_{G[l], V[l]}$.

The family $\{v_j\}_{j\in J}$ is a basis in $V[l]$. Let $\{x_j\}_{j\in J}$
be the dual basis in  $V[l]^*$.  The monomials of degree $d$ in $\{x_j\}_{j\in J}$
form the basis of the space  $k[V[l]]_d$
of homogeneous polynomials of degree $d$ on $V[l]$.
For $\alpha\in \{l=0\}$, the element $X_\alpha$
acts linearly on $k[V[l]]_d$, and it follows from the hypothesis that its matrix has rational coefficients in this basis. Since the algebra of  $G[l]$-invariants on
$V[l]$ coincides with $\{f\in k[V[l]]\mid
\operatorname{Lie}(G[l])\cdot f=~0\}$, this implies that
 this algebra is generated by certain homogeneous polynomials in $\{x_j\}_{j\in
J}$ with coefficients in $\mathbb Q$. By
\eqref{strat},  there is a nonconstant polynomial $f$ among them. It follows from the algebraic independence condition that
$f(v)\neq 0$. Hence, $v\notin {\cal
N}^{}_{G[l], V[l]}$.
\end{proof}

Note that the basis and the set mentioned in Proposition
\ref{represent} always exist because there exists a $\mathbb Z$-form of
$\mathfrak g$ defined by the Chevalley basis and an admissible $\mathbb
Z$-form of $V$ (see, for example, \cite{Hu}).
For the adjoint representation, these basis and set are given by the
$\mathbb Z$-form of the algebra~$\mathfrak g$.

\medskip

{\bf 2.9.}\hskip 1.5mm If the number of orbits in a null-cone is finite, then the strata do not necessarily coincide with orbits (see Example 3 in Section 4). It would be interesting to find out whether it is always possible turn strata into orbits in such cases by changing the inner product $\langle \ {,}\ \rangle$.
 If this is so, then our algorithm yields a classification algorithm for the nullforms of any visible representation \cite{K1}, \cite{K2}.

\section*{3. Algorithm}

First, we present a number of additional notations and definitions concerning geometry and certain special graphs.

\smallskip

 {\bf 3.1.}\ Let $E$    be a finite-dimensional linear Euclidean space over $\mathbb Q$ and let
 $M$ be its nonempty subset. Introduce the following notation (Fig.\,1):
 \begin{list}{}{\setlength{\topsep}{2mm}
 \setlength{\parsep}{-1mm}
\setlength{\itemindent}{0mm}}
\item[$\bullet$] ${\rm aff}\,M$ is the affine hull of the subset $M$ in $E$ (i.e., the minimal linear variety in $E$ that contains $M$);
\item[$\bullet$] ${\rm conv}\hskip .8mm M$
is the convex hull of $M$ in
$E$,
 \item[$\bullet$] ${\rm
perp}\hskip .8mm M$
is the perpendicular dropped from zero onto
${\rm aff}\, M$ (i.e., the unique vector from
${\rm aff}\,M$ that is orthogonal to the difference of any two vectors from ${\rm aff}\,M$).
 \end{list}

\begin{center}
\leavevmode \epsfxsize =10cm
\epsffile{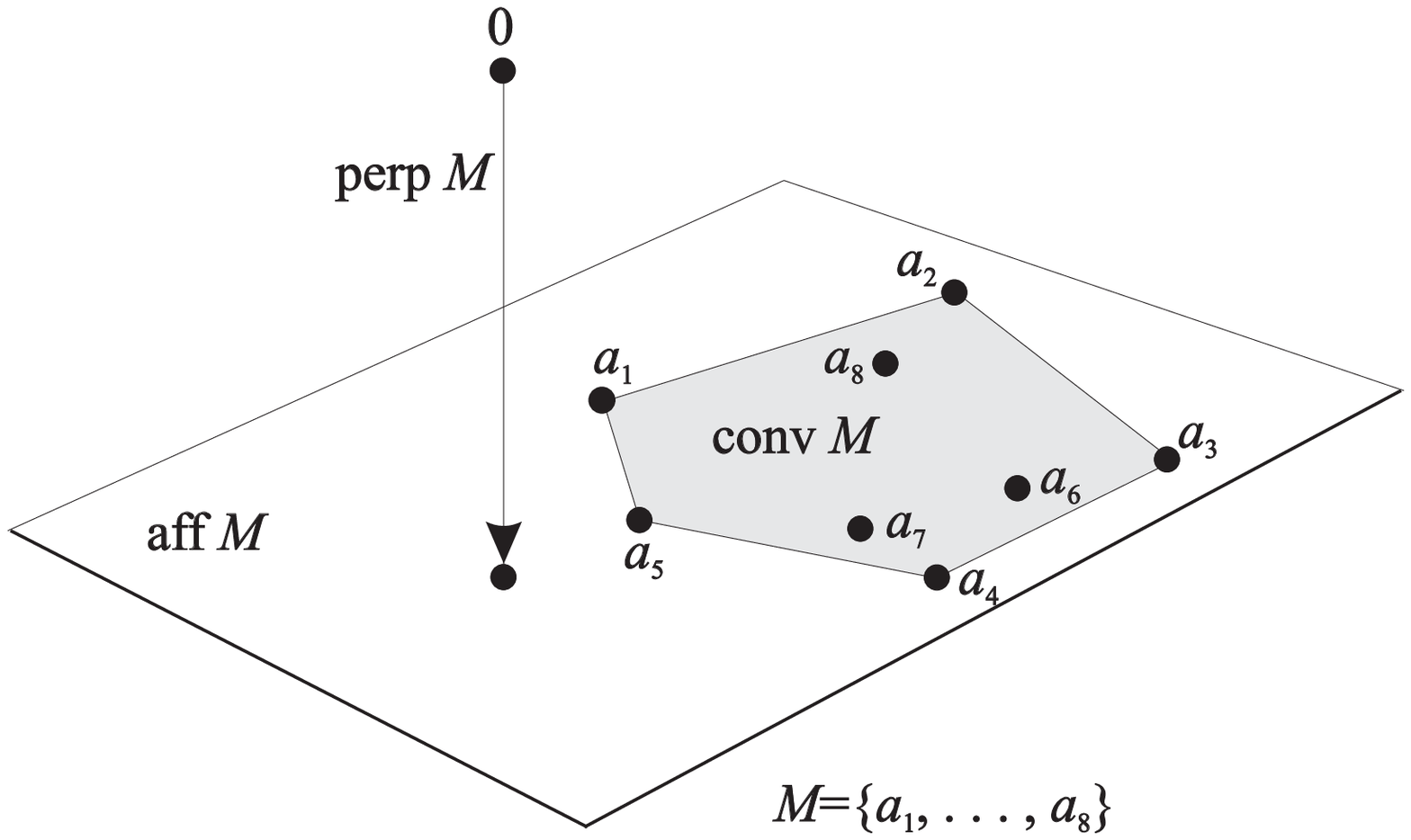} \centerline{\text{
Fig.\,1}}
\end{center}

{\bf 3.2.} \ Consider a rooted tree
$\Gamma$ (i.e., a connected graph without cycles and with a selected vertex called a root). Let $v_0,\ldots, v_d$  be the set of its vertices and let  $v_0$ be its root. 
 We endow each edge with an orientation (by specifying its beginning and end) directed ``from'' $v_0$,  i.e., so that any vertex different from the root is the end of precisely one edge (this requirement uniquely defines the orientation of edges). Figure 2 on the next page shows precisely such an orientation.

For any two different vertices $v_i$ and $v_j$
, there exists at most one route with the beginning at
$v_i$ and the end at $v_j$, that passes along edges. If such a route exists, we call it a
{\it geodesic} connecting
$v_i$ with $v_j$. The number of edges in a geodesic is called its {\it length}. For instance, the vertices $v_1$ and
$v_{20}$ in Fig.\,2 are not connected, but the vertices
$v_1$ and
$v_{21}$ are connected by the geodesic, and its length is 4.
 The root
$v_0$  is connected by a geodesic with any vertex.

A vertex is called the {\it end} vertex if it is not the beginning of any edge. For instance, the following vertices are end vertices in Fig.\,2:
 $v_{2}$, $v_{4}$,
$v_{6}$, $v_{8}$, $v_{10}$, $v_{11}$,
$v_{13}$, $v_{14}$, $v_{15}$, $v_{17}$,
$v_{18}$, $v_{20}$ ш $v_{21}$.

Each vertex is connected by a geodesic with at least one end vertex. The
{\it height} of a vertex $v_i$ is the greatest length of geodesics that connect $v_i$ with end vertices. For instance, in Fig.\,2 the root  $v_0$
has the height $5$, the vertex $v_3$
has the height $3$,  and the vertex $v_{11}$ has the height
$0$. The end vertices are exactly the vertices of height $0$.

\begin{center} \leavevmode
\epsfxsize =8.3cm \epsffile
{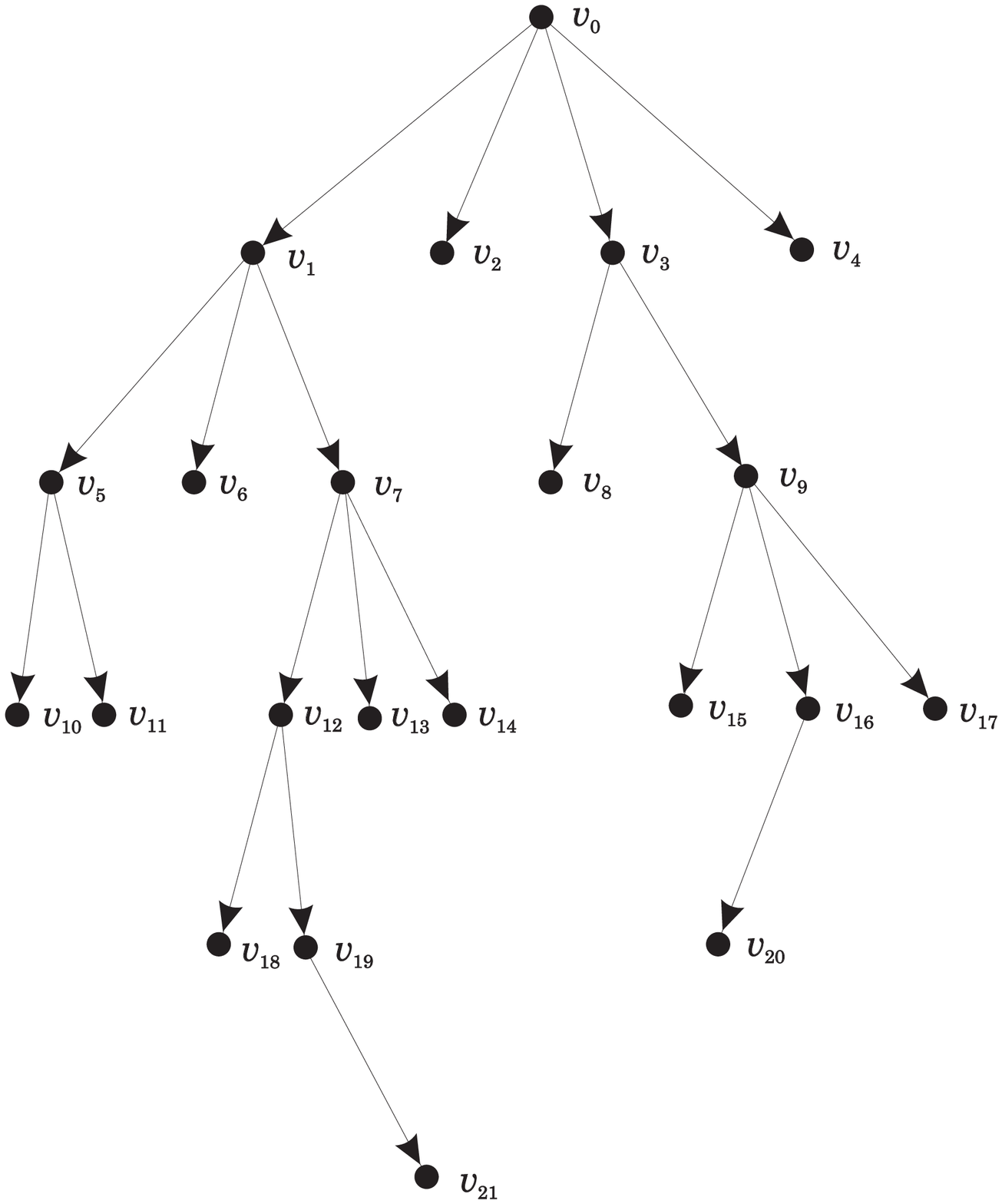}\vskip 4mm\centerline{\text{
Fig.\,2}}
\end{center}

If, moving along a geodesic $\gamma$,
we write out all successively encountered vertices, then the sequence of vertices obtained is called a {\it geodesic sequence} corresponding to
$\gamma$. For instance, if  $\gamma$
in Fig.\,2 is a geodesic connecting the root $v_0$ with the vertex $v_{21}$, then $\gamma$ corresponds to the the geodesic sequence of vertices
$v_0$, $v_1$, $v_7$, $v_{12}$, $v_{19}$,
$v_{21}$. Maximal geodesic sequences of vertices (i.e., those that cannot be embedded into geodesic sequences of greater length) correspond to $\gamma$ such that start at the root and end at an end vertex. For example, $v_0$, $v_3$, $v_8$
in Fig.\,2 is a maximal geodesic sequence of vertices.

\smallskip

 {\bf 3.3.}\hskip 1.5mm Passing to the construction of the algorithm, consider a nonzero element $l\in \mathfrak t(\mathbb
 Q)$. Suppose that the set of weights
ьэюцхёЄтю тхёют
 \begin{equation}
 \Delta^{}_{V, T}[l]:=\Delta^{}_{V, T}\cap
 \{l=1\}
 \label{dell}
 \end{equation}
that lies on the hyperplane $\{l=1\}$
is nonempty. By \eqref{vect}, this is equivalent to the condition $V[l]\neq
 \{0\}$.

Consider an orthogonal projection
 $$
 {\rm pr}^{}_l: \mathfrak t(\mathbb Q)\longrightarrow
 \mathfrak t[l](\mathbb Q).
 $$
 By \eqref{subtorus} and \eqref{vect},
the system of weights
 $\Delta^{}_{V[l], T[l]}$ of the torus $T[l]$ in
 $V[l]$ has the form
 \begin{equation}
 \Delta^{}_{V[l], T[l]}={\rm pr}^{}_l(\Delta^{}_{V,
 T}[l])
 \label{deltal}
 \end{equation}
and equality \eqref{deltal} holds with the mutiplicities of the weights
 taken into account if we assume that the latter are preserved under the projection
 ${\rm pr}^{}_l$.

\begin{prop} Let $M:= \Delta^{}_{V, T}[l]$.
 If $l$ is a stratifying element, then  {\rm(}see Fig.\;$3${\rm)}
\begin{gather}
{\rm perp}\hskip .8mm M\in {\rm conv} \hskip
.8mm M, \label{prop1}
\\
{\rm perp}\hskip .8mm M = {\rm perp}\hskip
.8mm \{l=1\}. \label{prop2}
\end{gather}
\end{prop}

\begin{center}
\leavevmode \epsfxsize =9.5cm
\epsffile{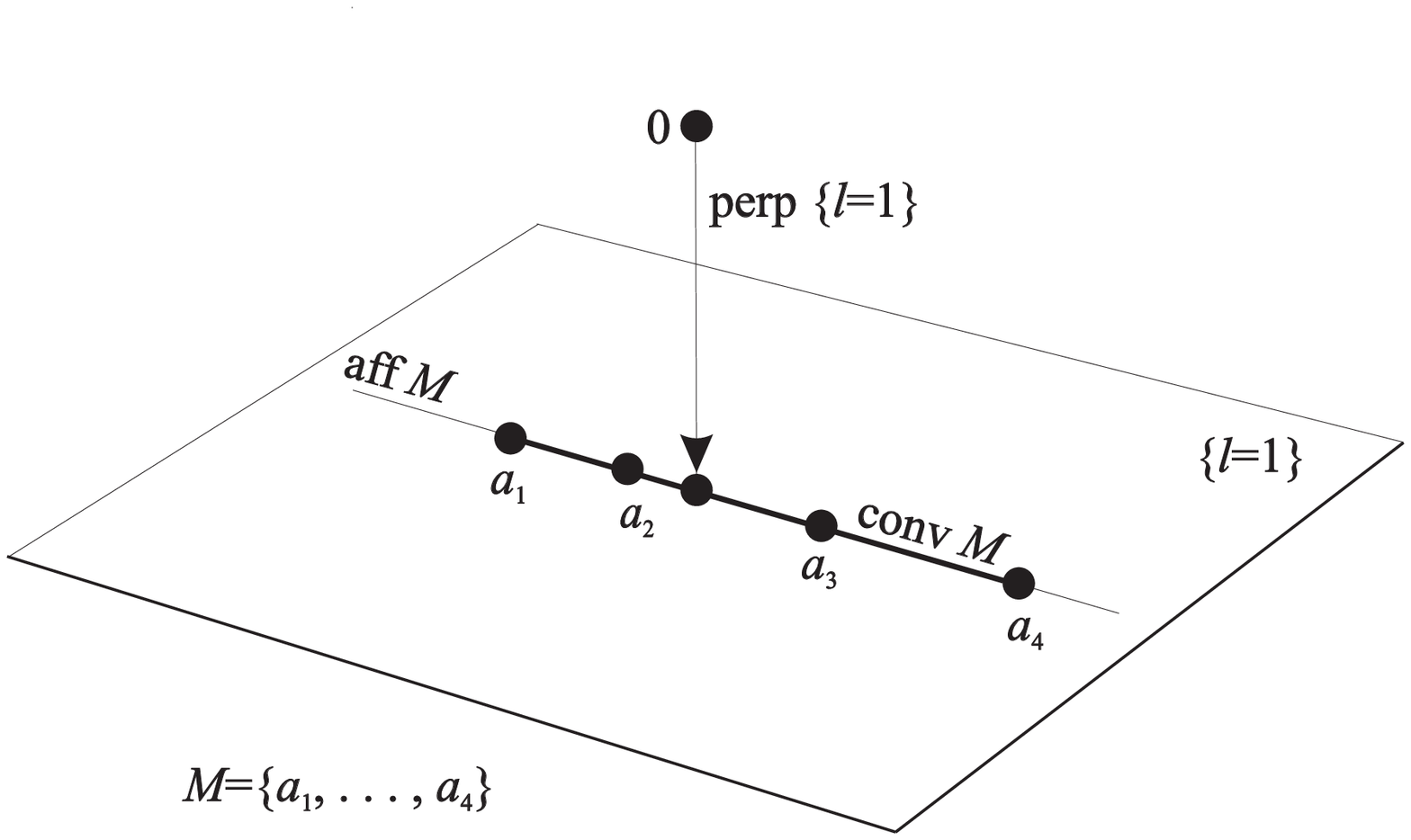} \centerline{\text
Fig.\,3}
\end{center}

\begin{proof}  If $l$ is a stratifying element, then
 \begin{equation}
 0\in {\rm conv}\hskip .8mm
 {\rm pr}^{}_l M.
 \label{incon}
 \end{equation}
Indeed, if inclusion \eqref{incon}  did not hold, then, by \eqref{deltal}, each $T[l]$-orbit in $V[l]$ would contain zero in its closure (see, for example,
 \cite[5.4]{PV}), and, hence,
 ${\cal N}^{}_{G[l],V[l]}=V[l]$ in spite of the fact that $l$ is a stratifying element.

Since the projection ${\rm pr}^{}_l$
commutes with  taking convex hull and the only vector of the hyperplane
$\{l=1\}$ that is mapped by  ${\rm pr}^{}_l$ to zero is ${\rm perp} \hskip
.8mm\{l=1\}$, inclusion \eqref{incon}
is equivalent to the inclusion
\begin{equation}
{\rm perp} \hskip .8mm\{l=1\}\in {\rm
conv}\, M. \label{incon1}
\end{equation}

Since $ {\rm aff}\hskip .8mm M \subseteq
\{l=1\}$, the vector ${\rm perp} \hskip
.8mm\{l=1\}$  is orthogonal to the difference of any two vectors from ${\rm aff}\hskip .8mm M.$ By the inclusion $ {\rm conv}\, M\subset {\rm
aff}\hskip .8mm M$ and the uniqueness of the perpendicular dropped from zero onto a linear variety, inclusion \eqref{incon1} implies \eqref{prop2}. In turn, \eqref{prop2} and
\eqref{incon1} imply \eqref{prop1}.
\end{proof}

{\bf 3.4.}\ We now introduce certain finite subsets in $\mathfrak
t(\mathbb Q)$ determined by the  $G$-module
$V$. Namely, consider the orbits of the Weyl group $W^{}_{G, T}$ on the set of all nonempty subsets of the system of weights
$\Delta^{}_{V, T}$ and fix a representative in each subset. Among these representatives, we choose those representatives $M$,
 that
 share the following properties:

 \begin{enumerate}[\hskip 5mm\rm(i)] \setlength{\itemsep}{-.9mm}
 \item
${\rm perp}\,M\neq 0$,
\item
${\rm
perp}\,M\in {\rm conv}\,M$,
\item
$M$
is the intersection of the system of weights
$\Delta^{}_{V, T}$  with the hyperplane
$\mathfrak t(\mathbb Q)$
passing through the end of the vector ${\rm perp}\, M$ and orthogonal to
this vector, i.e.,
$$M=\Delta^{}_{V, T}\cap \{l^{}_M=1\},
\hskip 1.5mm \text{ where }\
l^{}_M:=\frac{{\rm perp}\,M}{|\!|{\rm
perp}\,M|\!|^{2}},$$
\item
for
$l=l^{}_M$, inequality
\eqref{ineq} holds.
\end{enumerate}

When $M$ runs over the whole family of representatives
possessing
properties
(i)--(iv),  the vector $l^{}_M$ runs over a certain finite subset ${\mathfrak
L}^{}_{G, T, V}$ in $\mathfrak t(\mathbb Q)$.

If the $G$-module $V$ is trivial (i.e.,
$\Delta^{}_{V, T}=\{0\}$), then ${\mathfrak
L}^{}_{G, T, V}=\varnothing $ by (i). In particular, this is so if $G$ is the identity group $\{e\}$. If the $G$-module $V$ is nontrivial, then there exist nonzero nullforms in
$V$  (for example, vectors from $V_\mu$ for nonzero
$\mu\in \Delta^{}_{V, T}$)); hence, there are strata in
${\cal N}^{}_{G, V}$  that are different from $\{0\}$. According to the previous discussion, these strata are exhausted, without repetition, by varieties of the form ${\cal H}[l]$, where $l$
 is a stratifying vector from
${\mathfrak L}^{}_{G, T, V}$. In particular, ${\mathfrak L}^{}_{G, T,
V}\neq\varnothing$  in this case.

Thus, the algorithm for describing all strata of the null-cone ${\cal N}^{}_{G, V}$
of any nontrivial $G$-module $V$ will be obtained if we find an algorithm for verifying whether or not $l\in {\mathfrak L}_{G, T,
V}$  is a stratifying element.
That is what we shall now do.

\smallskip

{\bf 3.5.}\ To this end, we consider the following subset in
${\mathfrak L}^{}_{G, T, V}$:
\begin{equation} {\mathfrak M}^{}_{G, T,
V}:=\biggl\{l\in {\mathfrak L}^{}_{G, T, V}
\;\Big\vert\; \# \{l<0\}\cap \Phi_{\mathfrak g, T}\
=\hskip -1mm\sum_{\mu\in \{l< 1\}}\dim
V_{\mu}\biggr\}. \label{em}
\end{equation}

This subset ``governs'' the coincidence of ${\cal
N}^{}_{G, V}$ with $V$. Namely, if the
$G$-module $V$ is nontrivial, then \eqref{em} and Corollary
\ref{equal} of  Theorem \ref{stratum}
imply that the existence of a stratifying vector $l$ in ${\mathfrak
M}^{}_{G, T, V}$  is equivalent to the equality ${\cal N}^{}_{G,
V}=V$\!,  and that, for such $l$, the stratum ${\cal
H}[l]$  is open in $V$. Since different elements from
 ${\mathfrak
M}^{}_{G, T, V}$ lie in different $W_{G,
T}$-orbits, the openness of ${\cal H}[l]$ in
$V$ and claim (3) of Theorem \ref{stratum}
imply that {\it only one} such an element $l$ may exist. Note the particular case:
\begin{equation}
{\mathfrak M}^{}_{G, T, V}=\varnothing \
\Longrightarrow \ {\cal N}^{}_{G, V}\neq V.
\label{nothing}
\end{equation}
Implication \eqref{nothing} also holds for trivial $G$-modules $V$, since, for these modules, ${\cal N}^{}_{G, V}=\{0\}\neq V$ and
${\mathfrak M}^{}_{G, T, V}=\varnothing$
by virtue of ${\mathfrak L}^{}_{G, T,
V}=\varnothing$.

\smallskip

{\bf 3.6.}\ Henceforth, we assume that the $G$-module $V$
is nontrivial and pass to the algorithm for verifying whether or not $l\in {\mathfrak
L}^{}_{G, T, V}$ is a stratifying element. Since ${\cal N}^{}_{G[l], V[l]}$ is a closed subset in $V[l]$, condition
\eqref{strat} is equivalent to the fact that the element $l$ {\it satisfies} the inequality
\begin{equation}
\dim {\cal N}^{}_{G[l], V[l]}<\dim V[l].
\label{dime}
\end{equation}

In turn, because the number of strata is finite and these strata are open in their closures, inequality  \eqref{dime}  is equivalent to the fact that the dimension of all strata of the variety
${\cal N}^{}_{G[l], V[l]}$ is strictly less than
$\dim V[l]$.  As is shown above, the latter is equivalent to the fact that not a single element of the set ${\mathfrak M}_{G[l], T[l], V[l]}$
 determines the stratum of the null-cone ${\cal
N}^{}_{G[l], V[l]}$, i.e., {\it no element}
element $l'\in{\mathfrak M}_{G[l], T[l],
V[l]}$ {\it satisfies} the following inequality:
\begin{equation}
\dim {\cal N}^{}_{G[l][l'], V[l][l']}<\dim
V[l][l']. \label{dime1}
\end{equation}

Applying these arguments to the
$G[l][l']$-module $V[l][l']$, we obtain that
\eqref{dime1} {\it is not fulfilled} precisely when there exists an element $l''$ in the set
${\mathfrak
M}_{G[l][l'], T[l][l'], V[l][l']}$ that determi\-nes the stratum of the null-cone ${\cal N}^{}_{G[l][l'],
V[l][l']}$, i.e., such stratum that the following inequa\-li\-ty
{\it is fulfilled}:
\begin{equation*}
\dim {\cal N}^{}_{G[l][l'][l''],
V[l][l'][l'']}<\dim V[l][l'][l''].
\label{dime2}
\end{equation*}
An so on.

 \smallskip

{\bf 3.7.}\ This brings us to the consideration of all possible sequences of the form
\begin{equation}
l_1, \ldots, l_d, \quad \text{where }\ l_1=l,\
\text{and } \ l_{i+1}\in \mathfrak
M_{G[l_1]...[l_{i}], T[l_1]...[l_{i}],
V[l_1]...[l_{i}]} \ \, \text{for }
\,i=1,\ldots, d-1. \label{seq}
\end{equation}
By \eqref{subalg} and \eqref{vect},  for each of these sequences, the space $V[l_1]...[l_d]$
is nonzero and $G[l_1]...[l_d]$
is a connected reductive group of rank ${\rm rk}
G-d$.  Thus, $d\leqslant {\rm rk}\,
G$. The sequence \eqref{seq}
is maximal (i.e., nonextendable to the right) if and only if $\mathfrak
M_{G[l_1]...[l_{d}], T[l_1]...[l_{d}],
V[l_1]...[l_{d}]}=\varnothing$; in this case, it follows from \eqref{nothing} that
\begin{equation}
{\cal N}^{}_{G[l_1]...[l_{d}],
V[l_1]...[l_{d}]}\neq V[l_1]...[l_{d}].
\label{seqneq}
\end{equation}

{\bf 3.8.}\ Let us associate with vector $l\in
\mathfrak L^{}_{G, T, V}$ a rooted tree $\Gamma_l$, whose edges are directed ``from'' the root (see Subsection 3.2). This tree is successively constructed by the following procedure (see Fig.\,4 on the next page):

\begin{list}{}{\setlength{\topsep}{2mm}
 \setlength{\parsep}{-1mm}
\setlength{\itemindent}{0mm}}
\item[$\bullet$] $l$ is the root of $\Gamma_l$.
\item[$\bullet$] If the set
$\mathfrak M_{G[l], T[l], V[l]}$ is empty, then
$\Gamma_l$ consists only of the root.
\item[$\bullet$] If this set is nonempty and, say, consists of elements $a_1,\ldots,
a_s$, then these elements are vertices of
$\Gamma_l$, and exactly $s$
edges emanate from the root that have the ends $a_1,\ldots, a_s$.
\item[$\bullet$] Similarly, if the set
$\mathfrak M_{G[l][a_i], T[l][a_i],
V[l][a_i]}$ is empty, then there are no edges emanating from vertex $a_i$. \item[$\bullet$] If this set is nonempty and, say, consists of elements $b_1, \ldots, b_t$, then these elements are vertices of $\Gamma_l$ , and exactly  $t$  edges emanate from $a_i$ that have the ends
 $b_1, \ldots, b_t$.
\item[$\bullet$] And so on.
\item[$\bullet$] There are no other vertices and edges in the tree $\Gamma_l$.
\end{list}

\begin{center}
\leavevmode \epsfxsize =3.9cm \epsffile
{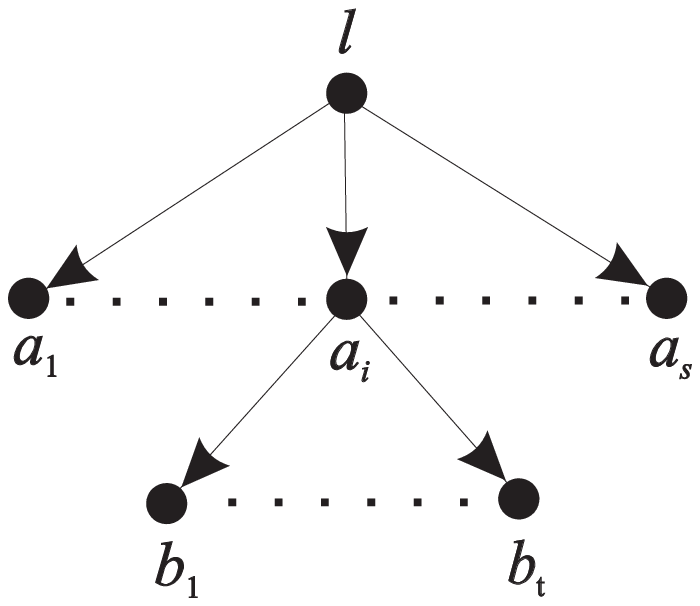} \vskip 1mm\centerline{\text
Fig.\,4}
\end{center}

{\bf 3.9.}\ The decision rule can be now formulated as follows.
\begin{thm}\label{signes} Let $l$ be an element of the set $\mathfrak L_{G, T, V}$. We assign plus or minus signs to the vertices of the tree $\Gamma_l$ {\rm (}see Fig.\,$5)$ by the following rule {\rm(}which uniquely defines this sign allocation{\rm):}
\begin{quote}
 A minus sign is assigned a vertex if and only if there exists an edge emanating from this vertex whose end has a plus sign {\rm(}in particular, all end vertices have a plus sign{\rm).}
\end{quote}
\indent Then, $l$
 is a stratifying element if and only if the root of the graph $\Gamma_l$
has a plus sign.
\end{thm}

\noindent{\bf Remark.} In fact, the proof and what has been said in Subsection 3.5
imply that the tree $\Gamma_l$
automatically possesses the following special property: {\it
at most one}  edge emanates from any vertex of $\Gamma_l$
whose other end has a plus sign. In particular, at most one edge leading to an end vertex emanates from any vertex.

\begin{center}
\leavevmode \epsfxsize =9.8cm \epsffile
{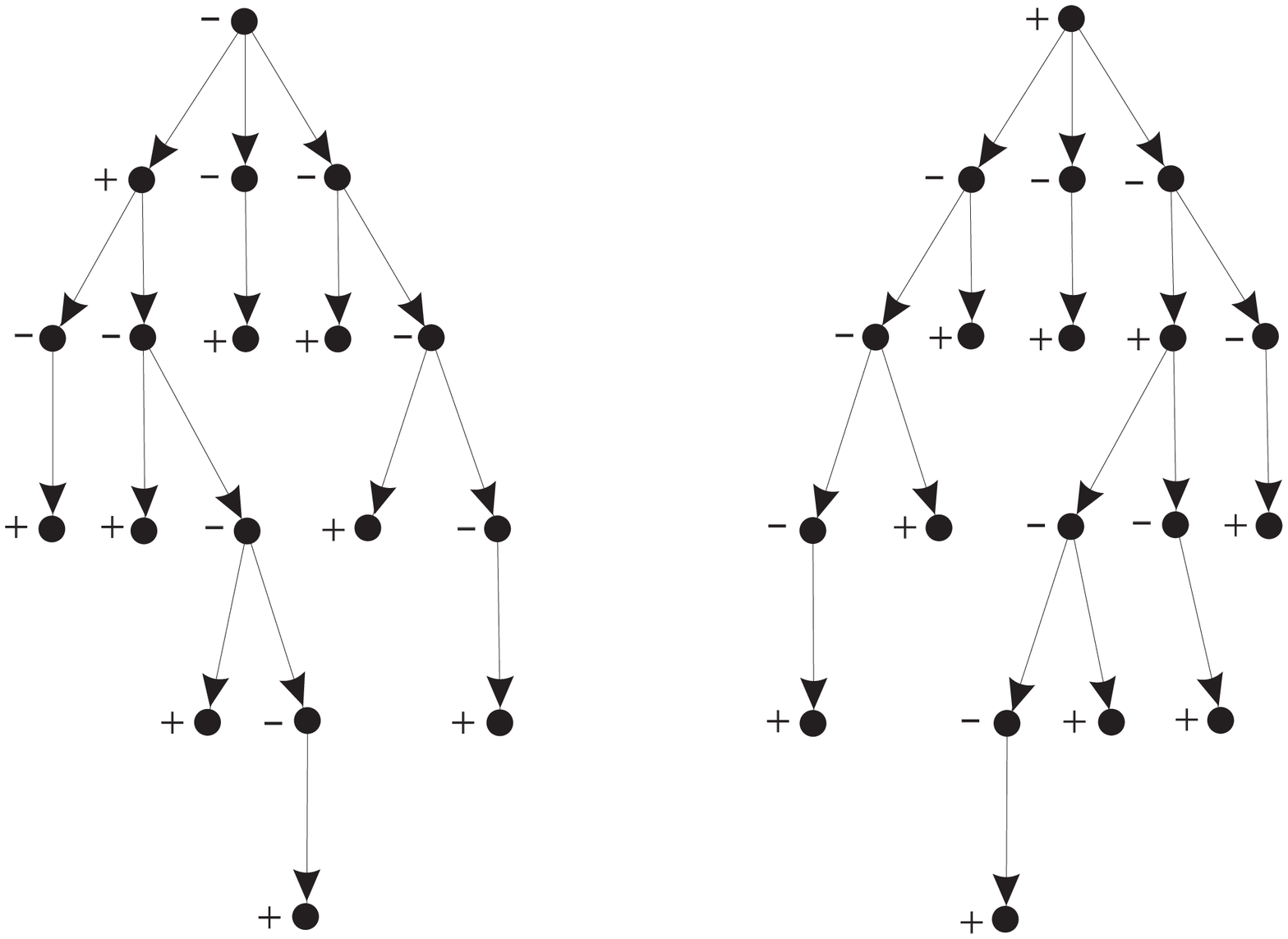} \vskip 1mm\centerline{\text
Fig.\,5}
\end{center}

\begin{proof} Sequences \eqref{seq}
are geodesic sequences of vertices of the tree $\Gamma_l$,
that start at the root. Consider a vertex $a$
and the geodesic $\gamma$ connecting the root
$l$ with $a$. Let $l_1, \ldots, l_d$, where
$l_d=a$,
be sequence \eqref{seq} corresponding to $\gamma$. Let us prove by induction on the height of the vertex $a$ that a plus sign is assigned the vertex $a$ if and only if condition \eqref{seqneq} holds. By the definition of the stratifying element, for a vertex of maximal height, i.e., for $a=l$, this gives the assertion of the theorem.

If $a$ has a zero height, then $a$
is an end vertex. Therefore, the sequence $l_1,
\ldots, l_d$ is maximal, and, hence, as was mentioned above, inequality
\eqref{seqneq} holds. On the other hand, by the sign allocation rule, the vertex
$a$ has a plus sign in this case. Thus, the assertion of the theorem holds for vertices of zero height.

Now, suppose that $a$ is a vertex of nonzero height. Then, sequence \eqref{seq}
is nonmaximal, i.e., the set $\mathfrak
M_{G[l_1]...[l_{d}], T[l_1]...[l_{d}],
V[l_1]...[l_{d}]}$ is nonempty. As was shown in Subsection 3.6, condi\-ti\-on \eqref{seqneq} holds if and only if any element $b\in\mathfrak M_{G[l_1]...[l_{d}],
T[l_1]...[l_{d}], V[l_1]...[l_{d}]}$ does not satisfy the condition
\begin{equation}
{\cal N}^{}_{G[l_1]...[l_{d}][b],
V[l_1]...[l_{d}][b]}\neq
V[l_1]...[l_{d}][b]. \label{seqneq1}
\end{equation}
On the other hand, it follows from the definition of the graph
$\Gamma_l$  that $b$ is its vertex and that the geodesic connecting the root with $b$ corresponds to the sequence $l_1, \ldots, l_d, b$. It follows from the definition of height that $b$ has a smaller height than $a$. Hence, by the induction hypothesis, condition
\eqref{seqneq1} is violated if and only if $b$ has a minus sign. Thus, condition \eqref{seqneq} is fulfilled if and only if all vertices
$b\in\mathfrak M_{G[l_1]...[l_{d}],
T[l_1]...[l_{d}], V[l_1]...[l_{d}]}$ have a minus sign. However, according to the sign allocation rule, the latter is equivalent to the fact that the vertex $a$ has a plus sign. This completes the proof.
\end{proof}

\setcounter{cor}{0}
\begin{cor}\label{noroots} Suppose that an element
 $l\in \mathfrak
L^{}_{G, T, V}$ is not orthogonal to any root. Then $l$ is a stratifying element.
\end{cor}

\begin{proof}
By \eqref{dell} and \eqref{deltal} and conditions
(ii) and (iii) in the definition of the set
$\mathfrak L^{}_{G, T, V}$ (see Subsection 3.4), zero belongs to $\operatorname{conv}
\Delta^{}_{V[l], G[l]}$. Therefore, for any nonzero element $l'\in V[l]$, the half-space $\{l'<1\}$ contains at least one weight. Since, by the hypothesis, there are no roots in $V[l]$
condition \eqref{em} implies that ${\mathfrak M}^{}_{G[l], G[l],
V[l]}=\varnothing$. Thus, the tree
$\Gamma^{}_l$  with the allocation of signs defined in Theorem~\ref{signes} ЁрёёЄрэютъющ чэръют
 has the form shown in Fig.\,6.
\begin{center}
\leavevmode \epsfxsize=2.4mm \epsffile
{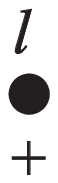} \centerline{\text Fig.\,6}
\end{center}
Hence, $l$
is a stratifying element.
\end{proof}

Since $\Phi^{}_{\mathfrak g,
T}=\varnothing$ for $G=T$, this and
\eqref{h}  imply the following corollary.

\begin{cor} \label{torus}{\rm (the case of a torus)}
If $G$ is a torus, $G=T$, then
\begin{list}{}{\setlength{\topsep}{1.4mm}
 \setlength{\parsep}{-1mm}
\setlength{\itemindent}{0mm}
\setlength{\labelwidth}{22mm}} \item[\rm(a)]
 all elements of the set $\mathfrak L^{}_{G,
T, V}$ are stratifying elements; \item[\rm (b)]
 the closures of the strata are linear subspaces of the form $\bigoplus_{\mu\in
\operatorname{conv} M}V_\mu$, where $M$
runs over those subsets in the system of weights
$\Delta^{}_{V, T}$, that possess properties {\rm(i)}, {\rm(ii)}, and {\rm(iii)} pointed out in Subsection {\rm (3.4)}. The strata themselves are complements, in these subspaces, to unions of finite sets of linear spaces of a similar form.
\end{list}
\end{cor}

\section*{4. Examples}

Denote by $F_{d,n}$ the space of forms of degree $d\geqslant 0$ in variables
$x_1,\ldots,x_n$ over $k$. We consider this space as the
$G$-module for $G=\operatorname{GL}_n$ or
$\operatorname{SL}_n$ with respect to the natural action by linear
transformations of variables.

We put $\mathbb N:=\{0, 1, 2,
\ldots\}$ and $\mathbb Z^{}_{+}:=\{1, 2, 3,
\ldots\}$.

\begin{example}[the case of rank 1] If $\dim T=1$,
then, by Corollary \ref{noroots} to Theorem
\ref{signes}, all elements of the set
$\mathfrak L^{}_{G, T, V}$
 are stratifying elements.

Suppose that, in view of Corollary \ref{torus} to
 Theorem \ref{signes}, the group $G$ is not a torus. Then, passing to the universal covering, we can assume that
$G=\operatorname{SL}_2$ and $T$
is the diagonal maximal torus in $G$.
The group $\operatorname{Hom}(T, k^\times)$
is generated by the character
$\operatorname{diag}(a, a^{-1})\mapsto a$.
Denote it by $\varepsilon$. Then,
$\Phi_{\mathfrak g, T}=\{2\varepsilon,
-2\varepsilon\}$, and, since $W^{}_{G,
T}=\{1, -1\}$, we can assume that
$\mathfrak L^{}_{G, T,
V}=\{\frac{\varepsilon}{\langle
\mu,\varepsilon\rangle}\mid \mu\in
\Delta^{}_{V, T}\cap \mathbb
Z^{}_{+}\varepsilon\}$. Since
$\operatorname{char}k=0$,  we have, up to an isomorphism,
$V=F^{}_{d_1,2}\oplus\ldots\oplus
F^{}_{d_s,2}$. Let $d^{}_{\text{odd}}$
(respectively, $d^{}_{\text{even}}$)  be the greatest positive odd (respectively, even) number among
$d_1,\ldots, d_s$, if such numbers exist, and 0
otherwise. Since
$\Delta^{}_{F_{d,2}, T}=\{d\varepsilon,
(d-2)\varepsilon, \ldots, (2-d)\varepsilon,
-d\varepsilon\}$ and all the weights in
$\Delta^{}_{F_{d,2}, T}$ have mumtiplicity $1$,  we obtain that the number of strata in ${\cal N}^{}_{G, V}$ that are different from
$\{0\}$
is equal to
$\left[\frac{d^{}_{\text{эхў}}+1}{2}\right]+
\frac{d^{}_{\text{ў╕Є}}}{2}$, and the dimension of the stratum ${\cal N}[l]$, defined by the element $l=\frac{\varepsilon}{\langle
m\varepsilon, \varepsilon\rangle}\in
\mathfrak L^{}_{G, T, V}$, is equal to
$1+\sum^{}_{d_i\geqslant m}\rceil m,
d_i\lceil$, where $\rceil a, b\lceil$ with $a,
b\in \mathbb Z_+$, $a\leqslant b$, denotes the number of integers in the interval $[a, b]$,  that have the same parity as $b$.

For example, when $s=5$ and $(d_1, d_2, d_2, d_4,
d_5)=(2, 3, 3, 4, 5)$, we have
$d^{}_{\text{эхў}}=5$, $d^{}_{\text{ў╕Є}}=4$
and $\mathfrak L^{}_{G, T, V}$ consists of
$\left[\frac{d^{}_{\text{эхў}}+1}{2}\right]+
\frac{d^{}_{\text{ў╕Є}}}{2}=5$ elements
$l_1,\ldots, l_5$.  Thus, there are exactly five strata that are different from $\{0\}$. The weights from $\Delta^{}_{V, T}$ are
 depicted by circles in Fig.\,7; the numbers at these circles are the multiplicities of these weights. The roots are marked by arrows. One can choose $l_i$ so that the point
$\{l_i=1\}$ is such as shown in Fig.\,7. The dimension of the stratum ${\cal
H}[l_i]$ is greater by 1 than the sum of multiplicities of the circles lying to the left of the circle
$\{l_i=1\}$; i.e., it is equal to 11, 8, 6, 3, and 2
for $i=1,\ldots, 5$, respectively. This, in particular, yields $\dim{\cal N}^{}_{G, V}=11$.
\begin{center}
\leavevmode \epsfxsize=13cm \epsffile
{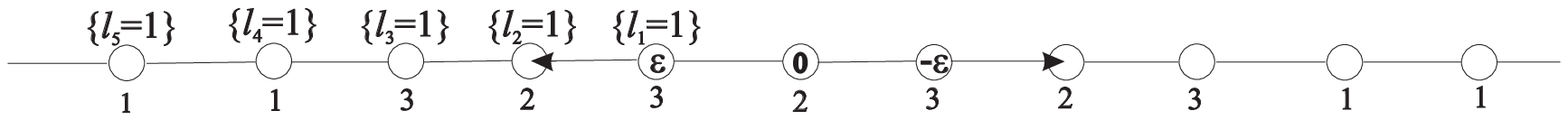} \centerline{\text Fig.\,7}
\end{center}
\end{example}

\begin{example}[the case of rank 2] Let $\dim
T=2$. By Corollary \ref{torus} to Theorem
\ref{signes}, we assume that $G$ is not a torus. The description of the straight lines $\{l=1\}$
for $l\in \mathfrak L^{}_{G, T, V}$
is equivalent, up to the action of the Weyl group, to the description of all possible straight lines $L$ on the plane
$\mathfrak t(\mathbb Q)$ that share the following properties:
 \begin{enumerate}[\hskip 5mm\rm(i)] \setlength{\itemsep}{-.9mm}
 \item $L$ does not pass through 0;
\item $L$ passes at least through one weight;
\item  the end of the perpendicular dropped from 0 onto $L$, is contained in a segment representing a convex hull of the weights lying on $L$.
    \end{enumerate}

If $L$ is such a line, then there are two possibilities: either $L$ is parallel to no root, or it is parallel to a certain root $\alpha$. In the first case, $l$
is a stratifying element by Corollary
\ref{noroots} to Theorem \ref{signes}. In the second case, it may happen that
$\mathfrak M^{}_{G[l], T[l], V[l]}\neq
\varnothing $. Namely, since
$\Delta^{}_{V, T}$ is invariant with respect to
$W^{}_{G, T}$  and, in particular, with respect to the reflection about a straight line orthogonal to $\alpha$,
 it follows from \eqref{em} that this occurs precisely when $L$ contains exactly two different single weights (in Fig.\,8, the weights are depicted by dark, and the roots, by light circles).
\begin{center}
\leavevmode \epsfxsize =6.5cm \epsffile
{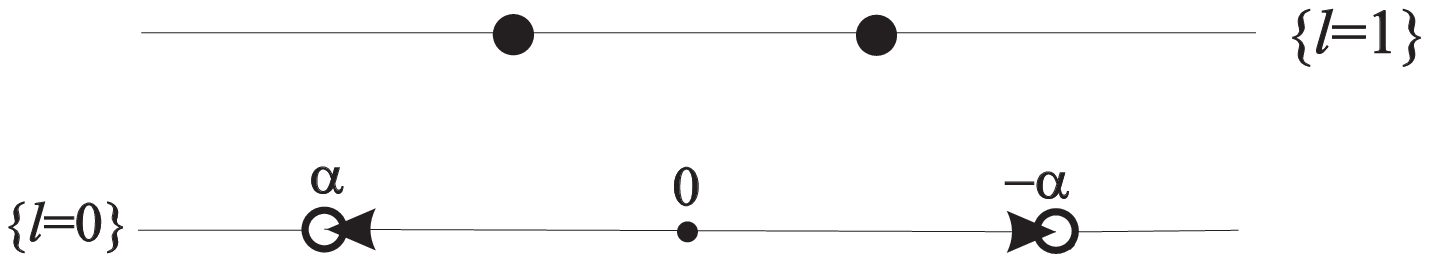} \centerline{\text Fig.~8}
\end{center}
In this case, $\mathfrak
M^{}_{G[l], T[l], V[l]}$ consists of a single element $l'$. Since, obviously, $\mathfrak M^{}_{G[l][l'],
T[l][l'], V[l][l']}\linebreak=\varnothing$, in this case the tree
$\Gamma_l$ with an appropriate sign allocation has the form shown in Fig.\,9.
\begin{center}
\leavevmode \epsfxsize =7.1mm \epsffile
{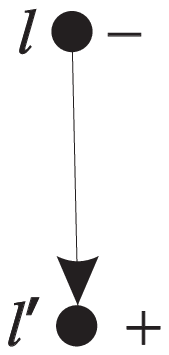} \centerline{\text Fig.~9}
\end{center}
Hence, by Theorem~\ref{signes},
$l$ is not a stratifying element.

\end{example}
\begin{example}[dependence on the norm; null-cones with a finite number of orbits]  Let $G={\rm
GL}_2$ and $V=F^{}_{1,2}\oplus
\wedge^2F^{}_{1,2}$.  There are exactly three nonzero $G$-orbits in $V$: ${\cal O}_1:=G\cdot
x_1,\  {\cal O}_2:=G\cdot (x_1\wedge x_2)$, and
${\cal O}_3:=G\cdot (x_1 + x_1\wedge x_2)$.
Let $T$ be the diagonal maximal torus in
$G$ and let $\varepsilon_i\in
\operatorname{Hom}(T, k^\times)$ be the weight of the variable $x_i$. The weights $\varepsilon_1$ and
$\varepsilon_2$ form a basis in the lattice
$\operatorname{Hom}(T, k^\times)$ and thereby in $\mathfrak t(\mathbb Q)$. The Weyl group
$W^{}_{G,T}$ has order 2 and acts by permutations of $\varepsilon_1$ and
$\varepsilon_2$. The sequence of vectors
$x_1$, $x_2$, $x_1\wedge x_2$ is a weight basis in $V$. Therefore, the system of weights has the form $\Delta^{}_{V, T}=\{\varepsilon_1,
\varepsilon_2,
\varepsilon_1+\varepsilon_2\}$  and  the multiplicity of every weight is $1$. The system of roots has the form
$\Phi^{}_{\mathfrak g,
T}=\{\varepsilon_1-\varepsilon_2,
\varepsilon_2-\varepsilon_1\}$. Positive definite $W^{}_{G, T}$-invariant inner products on $\mathfrak t(\mathbb
Q)$ are classified by the pairs of rational numbers $(a, b)\in\mathbb Q^2$,
satisfying the conditions $a>0$ and $a^2>b^2$:
to a pair $(a, b)$ is assigned the inner multiplication
for which $\begin{bmatrix}a & b\\b & a\end{bmatrix}$ is the Gram matrix of the basis
$\varepsilon_1$, $\varepsilon_2$.

Thus, in this case, the weights and the roots are depicted by dark and light circles, respectively (Fig.\,10); the left part of Fig.\,10 illustrates the case when $b>0$, while the right part, the case when
$b\leqslant 0$:

\begin{center}
\leavevmode \epsfxsize =9.8cm \epsffile
{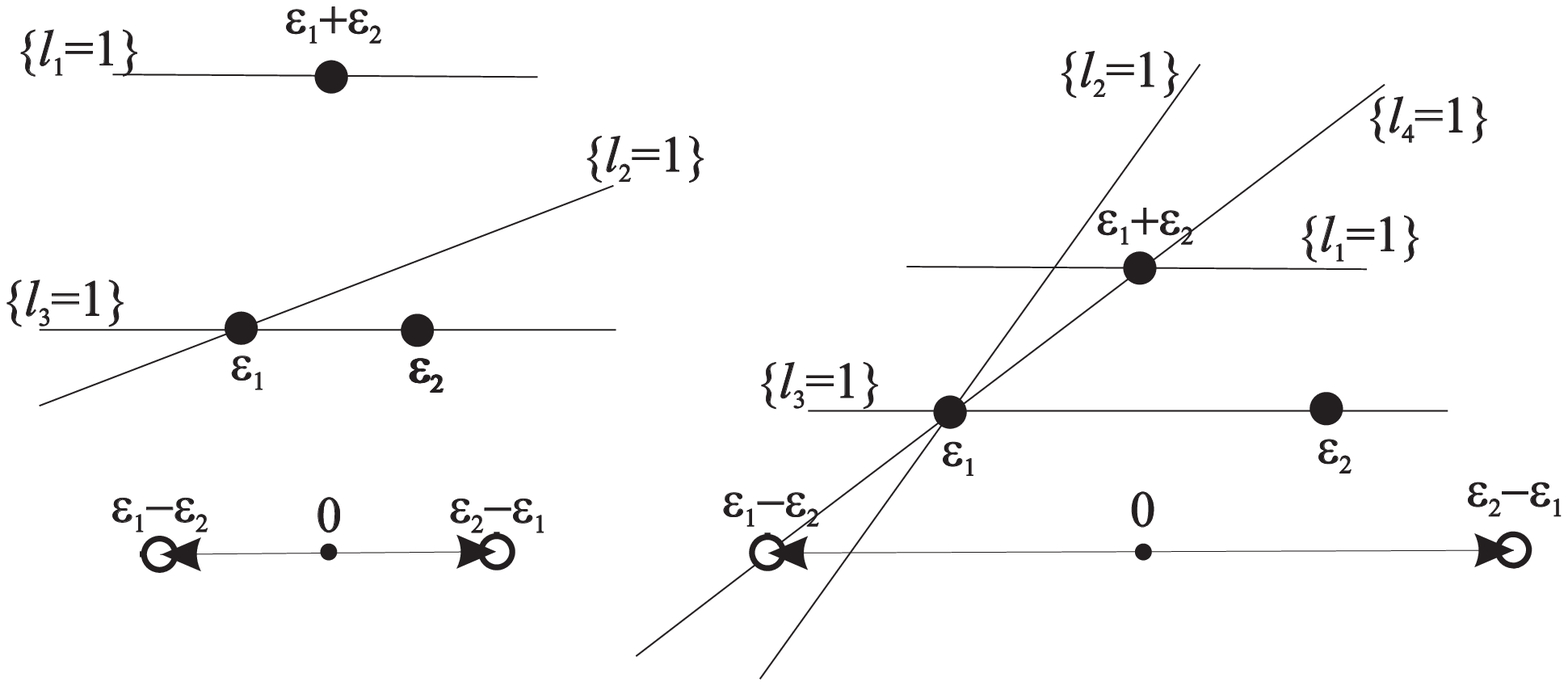} \centerline{\text Fig.\,10}
\end{center}

Figure 10 shows that the set
$\mathfrak L^{}_{G, T, V}$ consists of three elements $l_1, l_2$, and $l_3$ for $b>0$ and of four elements $l_1, l_2, l_3$, and $l_4$ for
$b\leqslant 0$,  and that, subjecting these elements, if necessary, to transformations from the Weyl group, one can assume that the straight lines $\{l_i=1\}$ are such as they are depicted in Fig.\,10. According to Example~2, if $b>0$, then there are exactly two strata in the null-cone
${\cal N}^{}_{G, V}$ that are different from $\{0\}$, namely, ${\cal
H}[l_1]$ and ${\cal H}[l_2]$; if
$b\leqslant 0$, then there are exactly three such strata, namely,
${\cal H}[l_1]$, ${\cal H}[l_2]$, and ${\cal
H}[l_4]$. It is clear that, if $b>0$, then ${\cal
H}[l_1]={\cal O}_2$, and ${\cal H}[l_2]={\cal
O}_1\cup {\cal O}_3$, whereas, if $b\leqslant 0$,
then ${\cal H}[l_1]={\cal O}_2$, ${\cal
H}[l_2]={\cal O}_1$ and ${\cal H}[l_4]={\cal
O}_3$.

Thus, in the general case, the strata depend on the norm chosen and, if the null-cone contains only a finite number of orbits, then some strata may even not be orbits (see \cite[Remark 4.10]{H2}).
\end{example}


\begin{example}
[ternary forms] Let $G={\rm SL}_3$ and
$V=F^{}_{d, 3}$, $d\geqslant 1$. Let $T$
be the diagonal maximal torus in $G$ and
$\varepsilon_i\in {\rm Hom} (T, k^\times)$
be the weight of the variable $x_i$. The space
$\mathfrak t(\mathbb Q)$ is two-dimensional. The weights
$\varepsilon_1$, $\varepsilon_2$, and
$\varepsilon_3$ generate the lattice ${\rm Hom}
(T, k^\times)$ , and the sum of these weights is equal to~$0$.
The system of roots $\Phi_{T, \mathfrak g}$
 is the set
$\{\varepsilon_i-\varepsilon_j\mid i\neq
j\}$. The set $\{\varepsilon_1,
\varepsilon_2, \varepsilon_3\}$ is preserved by the Weyl group $W_{G, T}$, and any permutation of the vectors $\varepsilon_1,
\varepsilon_2$, and $\varepsilon_3$ is performed by a certain element from $W_{G, T}$. Therefore,
$\parallel \varepsilon_1\parallel=\parallel
\varepsilon_2\parallel=\parallel
\varepsilon_3\parallel$, and
$\angle\{\varepsilon_i,
\varepsilon_j\}=2\pi/3$ for $i\neq j$,  which determines $\langle \ {,}\ \rangle$
uniquely up to proportionality. If $c_1, c_2, c_3\in
\mathbb Q$, then we denote the vector
$c_1\varepsilon_1 + c_2\varepsilon_2 +
c_3\varepsilon_3\in \mathfrak t(\mathbb Q)$
by $(c_1c_2c_3)$. Then, the system of weights is expressed as $\Delta^{}_{V,
T}=\{(c_1c_2c_3)\mid  c_1+c_2+c_3=d,\ c_i\in
\mathbb N\}$.  The multiplicity of the weight $(c_1c_2c_3)$
is equal to 1, and the set of monomials
$\{x_1^{c_1}x_2^{c_2}x_3^{c_3}\}$
 is linearly independent and spans $V_{(c_1 c_2c_3)}$.   Thus, the system of weights is given by the set of vertices of regular triangles with the side
$\sqrt{3}\parallel\varepsilon_i\parallel$
that form a partition of the regular triangle with vertices $(d00)$, $(0d0)$,
$(00d)$. In Fig.\,11 on the next page, the weights are depicted by dark circles and the roots, by light circles.

The straight lines $\{l=1\}$ for $l\in \mathfrak
L^{}_{G, T, V}$ and the stratifying elements in $\mathfrak L^{}_{G, T, V}$ are described according to Example 2. There is a unique, up to the action of the Weyl group, case when $l$ is not a stratifying element: in this case, the straight line $L_0:=\{l=1\}$ passes through the weights $(1(d-1)0)$ and
$(0(d-1)1)$.

\begin{center}
\leavevmode \epsfxsize =8.2cm \epsffile
{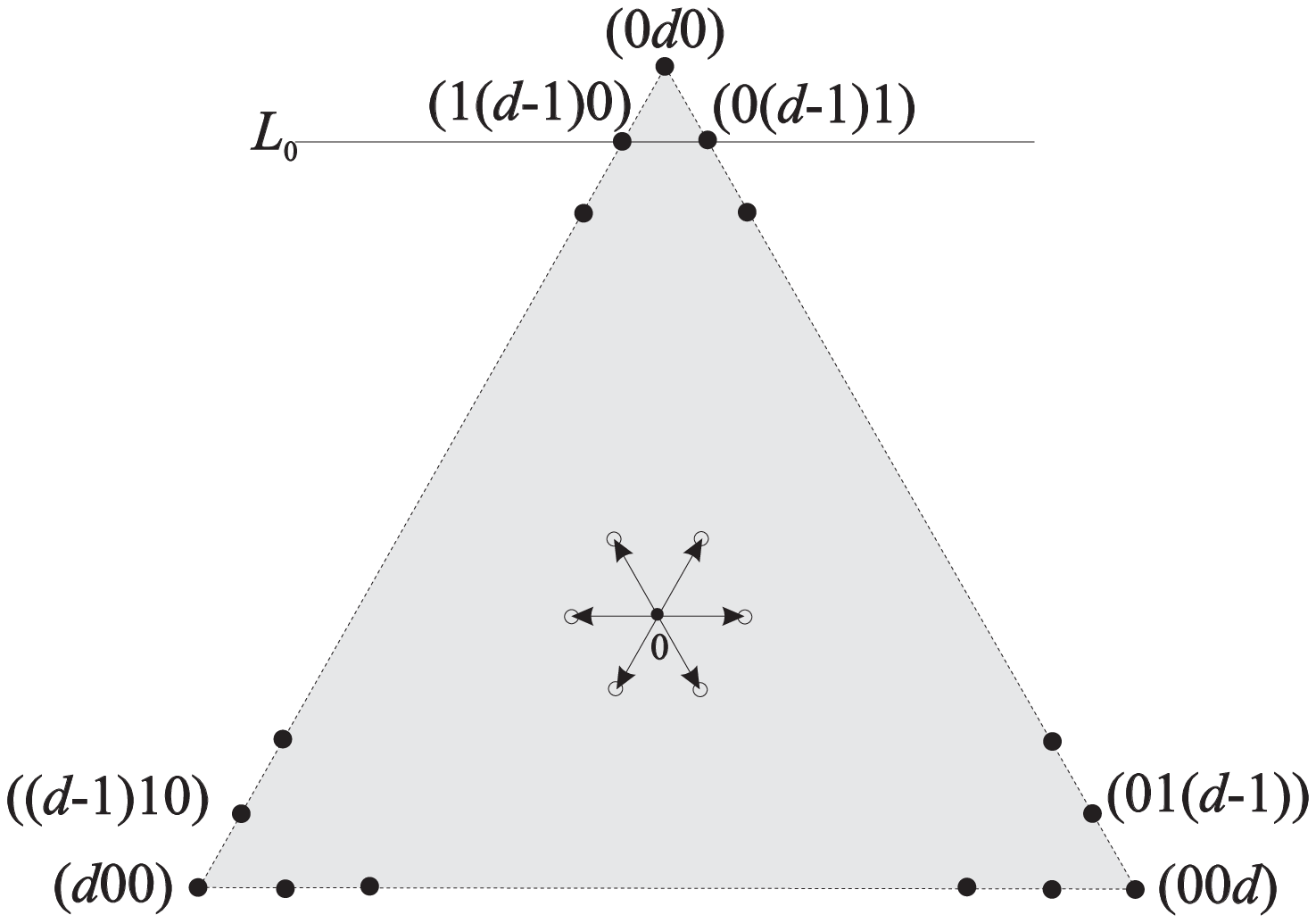} \centerline{\text Fig.11}
\end{center}

For example, in the case of ternary quartics (i.e., when $d=4$), one can easily verify that
$\mathfrak L_{G, T, V}$ contains exactly 12 elements $l_1,\ldots, l_{12}$ and that, subjecting them, if necessary, to transformations from the Weyl group, one can assume that the straight lines  $\{l_i=1\}$  are such as they are depicted in Fig.\,12.

\begin{center}
\leavevmode \epsfxsize =8.8cm \epsffile
{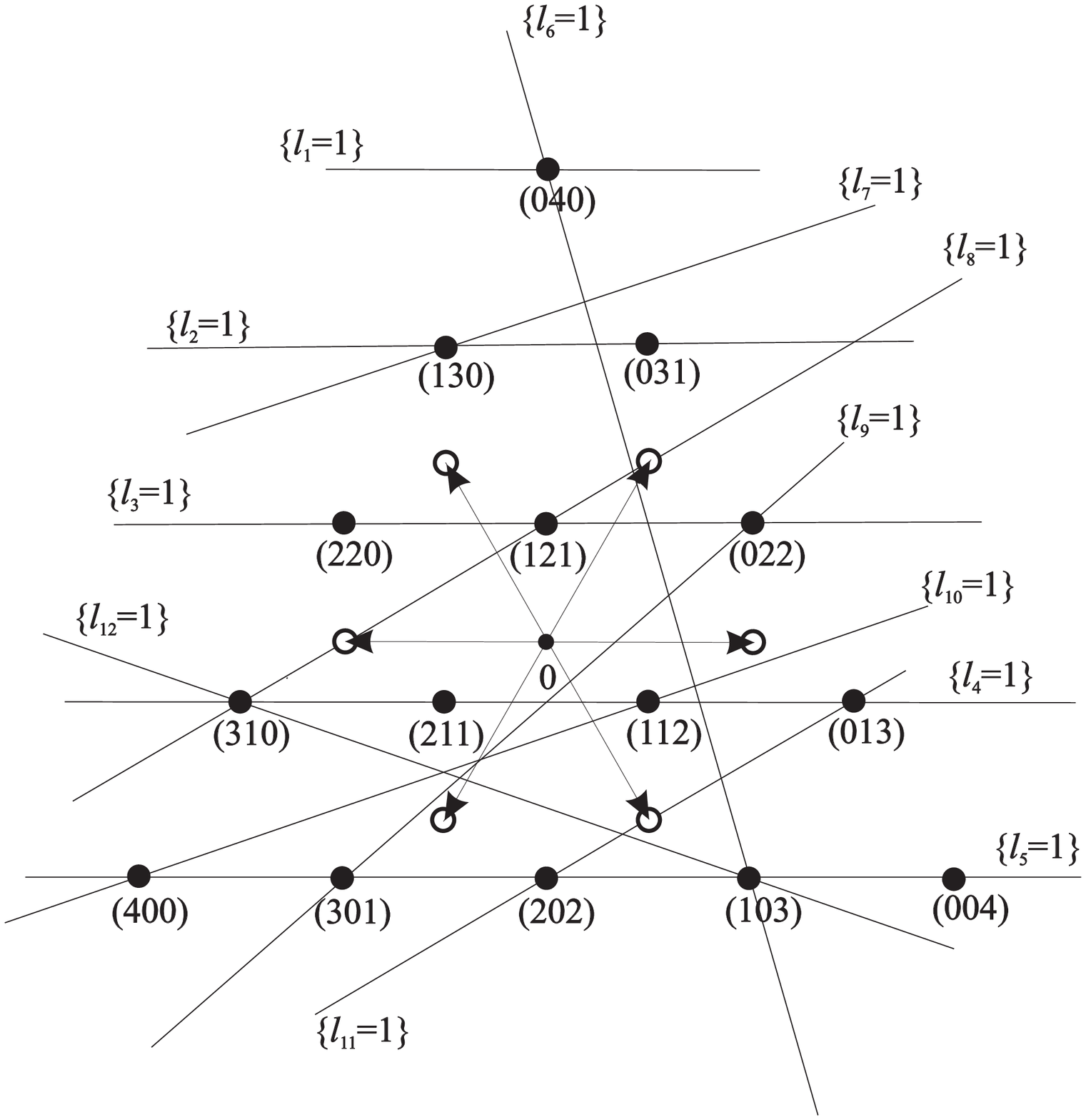} \vskip -2mm\centerline{\text Fig.12}
\end{center}

\noindent Here, only $l_2$ is not a stratifying element. Thus, there are exactly 11 strata in this case that are different from
$\{0\}$ , and these strata are given by ${\cal
H}[l_i]$, $1\leqslant i\leqslant 12$, $i\neq
2$. Formula \eqref{dim} gives their dimensions:
$\dim {\cal H}[l_i] = 3, 8, 11, 7, 9, 5, 9,
10, 10, 7$, and $8$ for $i=1, 3, 4,
5, 6, 7, 8, 9, 10, 11$, and $12$, respectively. Hence, $\dim{\cal
N}^{}_{G, V}=11$ and ${\cal N}^{}_{G, V}$ has a unique irredu\-ci\-b\-le component of maximal dimension (one can show that
${\cal N}^{}_{G, V}$ has another irredu\-ci\-b\-le component of smaller dimension
\cite{Hi}).
\end{example}

\begin{example}[nilpotent elements in reductive Lie algebras] If $V=\mathfrak
g$, then ${\cal N}^{}_{G, V}$ is the variety of all nilpotent elements of the Lie algebra $\mathfrak g$. It is well known \cite{H1},
\cite{Sl}
that, in this case, the strata of the null-cone ${\cal N}^{}_{G,
V}$ that are different from
$\{0\}$ are exactly the nonzero $G$-orbits
contained in
${\cal N}^{}_{G,
V}$ , i.e., the conjugacy classes of
nonzero nilpotent elements of the Lie algebra~$\mathfrak g$.
Hence, in this case our algorithm turns into the classification algorithm for the conjugacy classes of nilpotent elements in reductive Lie algebras
$\mathfrak g$, which is reduced to only simple geometric--combinatorial operations with the system of roots $\Phi_{\mathfrak
g, T}$. For the first time, such a classification was obtained in \cite{D} by a different method, which was later improved in \cite{BC}.  This classification is a particular case of the classification of the orbits of nullforms for the so-called
$\theta$-representations (see\,\cite{V3}). To obtain the latter classification, Dynkin's method was improved and extended in \cite{V1}, \cite{V2}, \cite{GV}  to the case of $\theta$-representations. A partial computer  implementation of this method was obtained in \cite{L2}.

We will illustrate the solution of this problem by means of our algorithm through the example of the simple exceptional Lie algebra $\mathfrak g$ of type ${\sf G}_2$.
In this case, the weights are depicted by dark circles in Fig.\,13.
\begin{center}
\leavevmode \epsfxsize =8.6cm \epsffile
{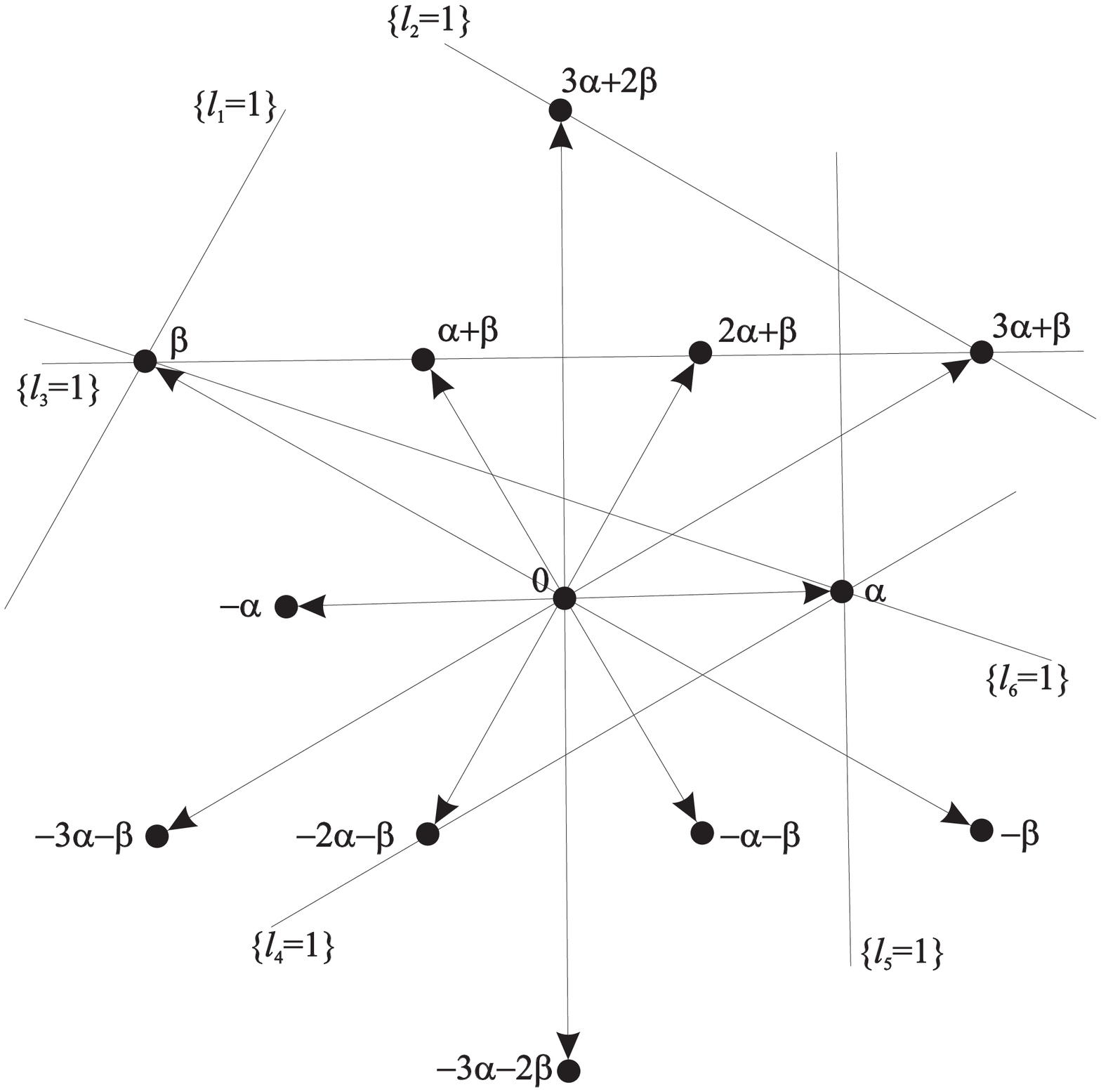} \centerline{\text Fig.\,13}
\end{center}
Here, the roots are nonzero weights. Using Fig.\,13, one can easily verify that the set $\mathfrak L_{G, T, V}$ contains exactly six elements $l_1,\ldots, l_{6}$ and that, subjecting them, if necessary, to transformations from the Weyl group, one can assume that the straight lines $\{l_i=1\}$ are as is shown in this figure. Of these elements, only $l_6$ is not orthogonal to any root; for all the other elements
$l_i$, the line $\{l_i=1\}$ contains exactly two roots precisely for $i=2$ ш 4.
Hence, the total number of strata different from
$\{0\}$ is four (therefore, there are precisely four nonzero conjugacy classes of nonzero nilpotent elements), namely,  ${\cal H}[l_i]$,
where $i=1, 3, 5$ and $6$.

The representatives of these four classes are obtained by means of what has been said in Subsection 2.8. Namely, it follows from Proposition
\ref{represent} that, if
$x_\gamma\in\mathfrak g_\gamma$
is an element of the Chevalley basis of the algebra $\mathfrak g$, then the representatives of the classes ${\cal H}[l_1]$,
${\cal H}[l_3]$,  ${\cal H}[l_5]$, and ${\cal
H}[l_6]$ are $x_\beta$,
$t_\beta x_\beta + t_{\alpha +\beta}
x_{\alpha +\beta} + t_{2\alpha +\beta}
x_{2\alpha +\beta} + x_{3\alpha +\beta}$,
$x_\alpha$, and $x_\alpha + t_\beta x_\beta$, respectively, where $t_\beta$, $t_{\alpha +\beta}$,
$t_{2\alpha +\beta}$ are elements from the field $k$ that are algebraically independent over~$\mathbb Q$.

 A simple additional argument shows that these representatives can be simplified. Namely, according to
\eqref{h} and \eqref{fiber}, if a
$G[l_i]$-orbit of a nonzero vector $v\in
V[l_i]$ is closed, then this vector lies in the class
${\cal H}[l_i]$. The group  $G[l_i]$ is a one-dimensional torus $T[l_i]$ for $i=6$ and is locally isomorphic to ${\rm SL}_2$ for $i\neq
6$. Therefore, in the first case, the closedness of
$G[l_i]\cdot v$ is equivalent to the fact that zero lies in the convex hull of the weights of the
$T[l_6]$-module $V[l_6]$ that occur in the decomposition of $v$ into a sum of weight vectors. This fact and Fig.\,13 imply that
$e_\alpha + e_\beta$ is a representative of the class ${\cal H}[l_6]$. In the second case, one can use the fact that, if the weights of certain weight vectors possess the following properties:
 \begin{enumerate}[\hskip 5mm\rm(a)] \setlength{\itemsep}{-1.mm}
 \item all these weights are different;
  \item zero the an inner point of their convex hull; and
   \item the difference of any two of these weights is not a root,
    \end{enumerate}
then the orbit of the sum of these weight vectors is closed (see\,\cite[Prop.
1.2]{DK}, \cite[╥хюЁхьр 6.19]{PV}). This is applicable to the case of $i=3$. Figure 13 shows that the roots of the group $G[l_3]$ with respect to the torus $T[l_3]$ are $\pm\alpha$, and the weights of
$V[l_3]$ are the orthogonal projections of $\beta$,
$\alpha+\beta$, $2\alpha+\beta$, and $3\alpha
+\beta$ onto the straight line $\{l_3=0\}$. Hence, the
$G[l_3]$-orbit of the vector $e_{\beta} +
e_{3\alpha+\beta}$ is closed. Therefore,
$e_{\beta} + e_{3\alpha+\beta}$  is a representative of the conjugacy class~${\cal
H}[l_3]$.

In conclusion, note that, for
$\theta$-representations, our algorithm turns into a classification algorithm for the orbits of nullforms. This claim results from the following proposition.

\begin{prop} For every $\theta$-representation, the strata of the null-cone are orbits.
\end{prop}
\begin{proof} The existence of a homogeneous version of the Morozov theorem
\cite[Theorem\;1]{V4} allows one to apply the arguments similar to those by which this fact was proved in
\cite[Prop.\;1]{Sl} for the adjoint representation (in turn, these arguments were inspired by \cite[\S1, Lemma
1.4]{K2}).
\end{proof}
\end{example}

\end{document}